\documentclass[11pt,reqno]{amsart}
\usepackage{amssymb}
\usepackage{mathptmx}
\usepackage{mathtools}
\usepackage{mathrsfs}
\usepackage[export]{adjustbox}
\usepackage[all]{xy}
\usepackage{stmaryrd}
\usepackage{fancyhdr}
\usepackage{hyperref}
\usepackage{tikz-cd}
\usepackage{cite}
\usepackage{amsmath,amsfonts}
\usepackage{graphicx}
\usepackage{wrapfig}
\usepackage{array}
\usepackage{amsthm}
\usepackage[left=1.2in, right=1.2in, top=1in, bottom=1in]{geometry}
\usepackage{etoolbox}
\patchcmd{\section}{\scshape}{\bfseries}{}{}
\makeatletter
\renewcommand{\@secnumfont}{\bfseries}
\makeatother
\xyoption{all}
\DeclareMathOperator{\Hom}{Hom}
\DeclareMathOperator{\Spec}{Spec}
\DeclareMathOperator{\sSpec}{Spec_s}

\DeclareMathOperator{\Frac}{Frac}

\DeclareMathOperator{\Pic}{Pic}

\DeclareRobustCommand{\coprod}{\mathop{\text{\fakecoprod}}}
\newcommand{\fakecoprod}{%
	\sbox0{$\prod$}%
	\smash{\raisebox{\dimexpr.9625\depth-\dp0}{\scalebox{1}[-1]{$\prod$}}}%
	\vphantom{$\prod$}%
}

\newcommand{\angles}[1]{\left\langle #1 \right\rangle}

\input xy
\xyoption{all}
\theoremstyle{definition}
\newtheorem{mydef}{\textbf{Definition}}[section]
\newtheorem{myeg}[mydef]{\textbf{Example}}

\newtheorem{rmk}[mydef]{\textbf{Remark}}

\theoremstyle{plain}
\newtheorem{mythm}[mydef]{\textbf{Theorem}}

\newtheorem*{nothma}{\textbf{Theorem A}}
\newtheorem*{nothmb}{\textbf{Theorem B}}
\newtheorem*{nothmc}{\textbf{Theorem C}}
\newtheorem*{nothmd}{\textbf{Theorem D}}
\newtheorem*{nothme}{\textbf{Theorem E}}

\newtheorem{lem}[mydef]{\textbf{Lemma}}
\newtheorem{pro}[mydef]{\textbf{Proposition}}

\newtheorem{cor}[mydef]{\textbf{Corollary}}

\newtheorem{condition}[mydef]{\textbf{Condition}}

\newcommand{\R}{\mathbb{R}}
\newcommand{\T}{\mathbb{T}}

\newcommand{\Z}{\mathbb{Z}}

\patchcmd{\abstract}{\scshape\abstractname}{\normalsize{\textbf{\abstractname}}}{}{}
\begin{document}

	\title{Vector bundles on tropical schemes}
	\author{Jaiung Jun}
	\address{Department of Mathematics, State University of New York at New Paltz, NY 12561, USA}
	\email{junj@newpaltz.edu}

	\author{Kalina Mincheva}
	\address{Department of Mathematics, Tulane University, New Orleans, LA 70118, USA}
	\email{kmincheva@tulane.edu}

	\author{Jeffrey Tolliver}
	\address{}
	\email{jeff.tolli@gmail.com}
	\makeatletter
	\@namedef{subjclassname@2020}{%
		\textup{2020} Mathematics Subject Classification}
	\makeatother
	
	\subjclass[2020]{14T10, 14C22}
	\keywords{monoid scheme, tropical geometry, vector bundle, semiring, semiring scheme, the field with one element, characteristic one, Picard group, scheme-theoretic tropicalization}
	\thanks{}
	
	\begin{abstract}
		We define vector bundles for tropical schemes, and explore their properties. The paper largely consists of three parts (1) we study free modules over zero-sum free semirings, which provide the necessary algebraic background for the theory, (2) we relate vector bundles on tropical schemes to topological vector bundles and vector bundles on monoid schemes, and finally (3) we show that all line bundles on a tropical scheme can be lifted to line bundles on a usual scheme in the affine case. 
	\end{abstract}
	
	\maketitle

	\section{Introduction}
	
	Tropical geometry is a recent area of mathematics and has had success and numerous applications including in algebraic geometry. It studies an algebraic variety $X$ over a valued field through its combinatorial counterpart, called the tropicalization $\textrm{trop}(X)$ of $X$. The combinatorial counterpart $\textrm{trop}(X)$ can be considered as a degeneration of the original variety $X$ and retains some of its invariants. Various authors have successfully answered questions of (enumerative) algebraic geometry by applying tools and ideas in tropical geometry, for instance, \cite{mikhalkin2005enumerative}, \cite{cools2012tropical}. 
	
	Algebraically, $\textrm{trop}(X)$ is described by polynomials with coefficients in an idempotent semiring, in particular, in the \emph{tropical semifield} $\mathbb{T}$ (see Example \ref{example: tropical semifield}). Geometrically, $\textrm{trop}(X)$ is a polyhedral complex with certain nice properties. The algebraic structure of the tropicalization is very degenerate so in order to prove analogues of classical algebro-geometric tools one needs to endow a tropical variety with extra structure. There are many different frameworks of scheme-theoretic foundations of tropical geometry among which blue schemes developed in \cite{lorscheid2015scheme} and the tropical ideals and tropical schemes introduced in \cite{maclagan2016tropical} and \cite{giansiracusa2016equations} respectively. The approach of J.~Giansiracusa and N.~Giansiracusa in \cite{giansiracusa2016equations} where  elegantly combines ideas from tropical geometry, and another emerging sub-field in algebraic geometry, namely, \emph{algebraic geometry over $\mathbb{F}_1$}. 
	
	The incarnation of algebraic geometry over $\mathbb{F}_1$ goes back to J.~Tits \cite{tits1956analogues}, where Tits found that the incidence geometry $\Gamma({\mathbb{F}_q})$ associated to a Chevalley group $G(\mathbb{F}_q)$ does not completely degenerate while the algebraic structure of $\mathbb{F}_q$ does as $q \to 1$. Tits suggested that the geometric structure $\lim_{q\to 1}\Gamma(\mathbb{F}_q)$ should be understood as a ``an incidence geometry defined over $\mathbb{F}_1$, the field with one element''. 
	
	Another motivation for algebraic geometry over $\mathbb{F}_1$ arises from a geometric approach towards the Riemann hypothesis, first considered by Y.~Manin in \cite{manin1995lectures}; one hopes to translate the geometric proof of the Weil conjectures to the case of number fields. One important step in this approach is
	interpreting $\Spec \mathbb{Z}$ as a curve over $\mathbb{F}_1$. As this set up is beyond the realm of classical scheme theory, one needs to enlarge the category of schemes. 
	Since its first appearance, there have been largely two approaches towards algebraic geometry over $\mathbb{F}_1$; (1) forgetting completely the additive structure (monoid schemes \cite{soule2004varietes}, \cite{deitmar2008f1},\cite{con2} for instance), and (2) maintain, but weaken the additive structure (semiring schemes \cite{oliver1}, \cite{oliver2}, \cite{giansiracusa2016equations} or more generally semiring-valued sheaves on topoi \cite{connes2014arithmetic}, \cite{connes2017geometry}). 
 
 The last 5 papers have a strong relation to previously developed $\mathbb{F}_1$-geometry. They define a rich geometric structure (scheme, topos) in the semiring setting, often arising as a scalar extension of $\mathbb{F}_1$ (monoid) schemes. There is also an $\mathbb{F}_1$ approach with analytic flavor, developed in \cite{Bambozzi}.

	Tropical schemes are a special class of semiring schemes. To a scheme $X$ embedded in a toric variety defined over a valued field, one can associate a tropical scheme $\textrm{Trop}(X)$, called the scheme-theoretic tropicalization of $X$ whose set of $\mathbb{T}$-rational points is the (set-theoretic) tropicalization $\textrm{trop}(X)$ of $X$.

	In \cite{jun2019picard}, we explored several properties of Picard groups for tropical toric schemes. Our main motivation was to take the first step toward formulating a scheme-theoretic tropical Riemann-Roch theorem. In particular, we showed that the Picard group of a monoid scheme $X$ is stable under the scalar extension to the tropical semifield $\mathbb{T}$ under certain technical conditions.
	
	In this paper, we initiate the study of vector bundles on semiring schemes (and hence tropical schemes). Properties of vector bundles on semiring schemes are rather subtle and sometimes counter-intuitive as they reflect the subtleties of linear algebra over idempotent semirings. For instance, any free module over the \emph{Boolean semifield} $\mathbb{B}$ (Example \ref{example: boolean}) has a unique basis (Example \ref{example: basis over B}). In fact, we prove that ``few bases exist'' over idempotent semirings; see Proposition \ref{proposition: unique basis} for the precise statement. This is a big difference with the classical theory of free modules.

	In \cite{pirashvili2015cohomology}, I.~Pirashvili provides a cohomological description of the vector bundles on a monoid scheme, and proves that for a connected separated monoid scheme $X$, any vector bundle is a coproduct of line bundles. In \S\ref{section: vector bundles for tropical toric schemes}, we first prove that indeed a similar result holds for semiring schemes which satisfy a certain local condition:

	\begin{nothma}(Theorem \ref{theorem: vector bundle decomposition})
		Let $X$ be an irreducible semiring scheme which is locally isomorphic to $\Spec R$, where $R$ is a zero-sum free semiring (Definition \ref{definition: ideals for semirings}) with only trivial idempotent pairs (Definition \ref{definition: idempotent pair}). Then any vector bundle of rank $n$ on $X$ is a coproduct of $n$ copies of line bundles on $X$. Moreover, this decomposition is unique up to permuting summands. 
	\end{nothma}
	
	The following theorem generalizes our previous result on line bundles, showing that under certain technical conditions, the set of vector bundles of a monoid scheme $X$ is stable under the scalar extension to any idempotent semifield:
	
	\begin{nothmb}(Theorem \ref{theorem: bundles stable under scalar extension})
		Let $X$ be an irreducible monoid scheme satisfying Condition \ref{condition: condition on open cover}, and $K$ be an idempotent semifield. Then, there exists a natural bijection between $Vect_n(X)$ and $Vect_n(X_K)$. 
	\end{nothmb}

	In \S \ref{section: topological bundles}, we introduce the notion of topological $\mathbb{T}$-vector bundles, and then prove the following, which is analogous to Theorem A, but also implies that in this setting line bundles are trivial.

	\begin{nothmc}(Theorem \ref{theorem: main theorem of topological bundles})
		Let $X$ be a connected paracompact Hausdorff space.  Then there is a canonical split surjection from the set of isomorphism classes of topological $\mathbb{T}$-vector bundles of rank $n$ on $X$ to the set of isomorphism classes of $n$-fold covering spaces of $X$.  Furthermore, a topological $\mathbb{T}$-vector bundle is trivial if and only if its associated covering space is trivial.
	\end{nothmc}

	In \S \ref{section: labelled k-algebras} and \S \ref{section: Detropicalization of line bundles}, we recall and utilize \emph{labelled $K$-algebras} (Definition \ref{definition: labelled algebra}) to study line bundles on tropical schemes in relation to those on ordinary schemes. In fact, the adoption of labelled algebras enables us to lift line bundles on a tropical scheme to an ordinary scheme in a suitable way. 
	
	Let $K$ be a valued field whose value group is $\mathbb{R}$ and $\mathcal{O}_K$ be the associated valuation ring. A labelled $K$-algebra is a $K$-algebra $A$ equipped with an epimorphism $\phi:K[M] \to A$ for some monoid $M$ such that $\phi|_M$ is injective. The extra structure $\phi$ allows one to perform the scheme-theoretic tropicalization for $\Spec A$ with respect to the map $\phi$. We note that the idea of labelled algebras is not new; a similar idea has been implemented in tropical geometry as a tool to study scheme-theoretic tropicalization, for instance, see \cite{lorscheid2015scheme} and \cite{giansiracusa2014universal}. 
	
	Another key player in \S \ref{section: labelled k-algebras} and \S \ref{section: Detropicalization of line bundles} is the semiring of finitely generated submodules. To be precise, we work in the following setting: Let $A$ be a $K$-algebra and $\mathbb{S}_\textrm{fg}(A)$ be the set of finitely generated $\mathcal{O}_K$-submodules of $A$. The set $\mathbb{S}_\textrm{fg}(A)$ is naturally equipped with a semiring structure; $M+N:=\angles{M \cup N}$ and $MN:=\angles{MN}$ for finitely generated $\mathcal{O}_K$-submodules $M$, $N$.  The algebraic structure of $\mathbb{S}_\textrm{fg}(A)$ is similar in many respects to that of $A$.  For instance they have the same Zariski topology (Theorem \ref{theorem: topology of space of submodules}).  In a similar vein, \cite{tolliver2016extension} proves the classical extension of valuations theorem for a field extension $L/K$ by reducing it to the problem of extending the valuation on $\mathbb{S}_\textrm{fg}(K)$ to one on $\mathbb{S}_\textrm{fg}(L)$.
	
	%We first prove that for the set of rational points over a $\mathbb{T}$-algebra $S$ of the scheme-theoretic tropicalization of an affine scheme bijectively correspond to the set of monomial valuations with values in $S$, the scheme-theoretic tropicalization of an affine scheme is closely related to $\mathbb{S}_\textrm{fg}(A)$. To be precise, we prove the following:
	
	Equipped with the above, in \S \ref{section: labelled k-algebras} we prove that for a labelled $K$-algebra $A$ the semiring $\mathbb{S}_{\mathrm{fgmon}}(A)$ of finitely generated monomial $\mathcal{O}_K$-submodules of $A$ is the same thing as the scheme-theoretic tropicalization of $A$. To be precise, we prove the following. 
	
	\begin{nothmd}(Proposition \ref{proposition: monomial valuation tropicalization} and Corollary \ref{corollary: tropical interpretation of monomial submodules})
		Let $(A,\phi)$ be a labelled $K$-algebra.  Let $S$ be a $\mathbb{T}$-algebra.  Let $\mathrm{Trop}(A)$ be the coordinate semiring of the scheme-theoretic tropicalization of $\Spec A$ with respect to $\phi:K[M] \to A$. Then, there is a natural one-to-one correspondence:
		\[
		\{\textrm{homomorphisms }     \mathrm{Trop}(A) \rightarrow S \}	\longleftrightarrow \{\textrm{monomial valuations on $A$ with values in $S$} \}.
		\] 
		Furthermore, $\mathrm{Trop}(A) \simeq \mathbb{S}_{\mathrm{fgmon}}(A)$ as $\mathbb{T}$-algebras.
	\end{nothmd}

	Finally in \S \ref{section: Detropicalization of line bundles}, we show how one can obtain line bundles on $X=\Spec A$, where $A$ is a finitely generated $K$-algebra, from those on the scheme-theoretic tropicalization of $X$. Recall that for a semiring $S$, we let $\sSpec S$ be the subset of $\Spec S$ consisting of prime ideals which are saturated (Definition \ref{definition: ideals for semirings}), and impose the subspace topology. We first enrich $\sSpec S$ to a semiringed space. This is an essential step to lift line bundles on $\sSpec S$ to $\Spec A$ when $S= \mathbb{S}_\textrm{fg}(A)$. Then, we prove the following.

	\begin{nothme}(Theorem \ref{theorem: picard isomorphism})
		Let $A$ be a $K$-algebra without zero-divisors and is finitely generated as a $K$-algebra. Then one has an isomorphism
		\[
		\Pic(\Spec A) \simeq \Pic(\sSpec \mathbb{S}_\mathrm{fg}(A)).
		\] 
		
	\end{nothme}
	
	%In particular, by Theorems D and E, one can define the following two maps (Definition \ref{definition: lifting def}):
	%\[
	%\tau_s: \Pic(\sSpec \mathbb{S}_\mathrm{fgmon}(A))\rightarrow \Pic(\Spec A), \quad \tau: \Pic(\Spec \mathbb{S}_\mathrm{fgmon}(A))\rightarrow \Pic(\Spec A).
	%\]
	
	This paper is organized as follows. In \S \ref{section: preliminaries}, we collect basic notions and properties of monoids and semirings which are used in later sections. In \S \ref{section: Free modules and projective modules over semirings}, we explore several technical properties of free and projective modules over semirings. By appealing to these results, in \S \ref{section: vector bundles for tropical toric schemes}, we introduce the notion of vector bundles for tropical schemes, and investigate their properties, and prove our main first main theorems. In \S \ref{section: topological bundles}, we study topological $\mathbb{T}$-bundles. In \S \ref{section: labelled k-algebras}, we introduce the notion of labelled $K$-algebras, and prove our Theorem D. Finally, in \S \ref{section: Detropicalization of line bundles}, we prove Theorem E and introduce lifting maps.

	\bigskip
	
	\textbf{Acknowledgment} J.J. was supported by an AMS-Simons travel grant. The authors would like to thank the anonymous referee for their comments and suggestions in improving the manuscript.

	\section{Preliminaries}\label{section: preliminaries}
	
	In this section, we briefly recall basic definitions and properties of monoid schemes and semiring schemes (also tropical schemes as a special case) which play a key role in the later sections. Most of the material in this section can be found in \cite{giansiracusa2016equations}, \cite{jun2017vcech}, \cite{jun2019picard}, \cite{jun2020lattices}.

	\subsection{Monoid schemes}
	
	In this paper, by a monoid, we will always mean a commutative (multiplicative) monoid $M$ with an absorbing element $0_M$, that is, $0_M\cdot a=0_M$ for all $a \in M$. The main reason that we prefer to work with monoids $M$ with $0_M$ is that they seem to behave more nicely. For monoids $M$ without $0_M$, some pathological examples may arise, for instance a free $M$-module is not flat in general in this case. See \cite[\S 2.1]{pirashvili2015cohomology} for some discussion. 
	
	Let $M$ be a monoid. An nonempty subset $I\subseteq M$ is said to be an \emph{ideal} if $MI \subseteq I$. An ideal $I$ is said to be \emph{prime} if $M- I$ is a multiplicative nonempty subset of $M$, and \emph{maximal} if $I$ is a proper ideal (i.e., $I \neq M$) which is not contained in any other proper ideal.

	The \emph{prime spectrum} $\Spec M$ of a monoid $M$ is the set of all prime ideals of $M$ equipped with the Zariski topology. As in the prime spectrum of a commutative ring, the set $\{D(f)\}_{f \in M}$ forms an open basis of $\Spec M$, where $D(f):=\{\mathfrak{p} \in \Spec M \mid f \not\in \mathfrak{p}\}$. We also let $V(f):=D(f)^c$. 
	
	One can construct affine monoid schemes via prime ideals and the Zariski topology by appropriately constructing the structure sheaf for $X=\Spec M$ mimicking the case of rings, or one can simply define the category of affine monoid schemes as the opposite category of monoids. Since a monoid scheme can be also understood as a functor of points, these two definitions are equivalent, i.e., a monoid scheme is a functor which is locally representable by monoids. A monoid scheme is then defined to be a topological space equipped with a structure sheaf of monoids locally isomorphic to affine monoid schemes. We refer the reader to \cite[\S 2]{jun2019picard} for the precise definitions and details. 
	
	Let $M$ be a monoid and $R$ be a ring. Here we take the naive approach and interpret $\mathbb{F}_1$ as a pointed multiplicative monoid with one element which we call 1 and absorbing element 0;  alternatively, $\mathbb{F}_1$ is an initial object in the category of (pointed) monoids. We define $M\otimes_{\mathbb{F}_1}R$ to be the monoid ring $R[M]$, where we identify $0_M$ and $0_R$. 
 This defines a functor from the category of monoids to the category of $R$-algebras, sending $M$ to $ R[M]$. There is also a natural functor (forgetful functor) $\mathcal{F}$ from the category of rings to the category of monoids sending a ring $A$ to $(A,\times)$, the underlying multiplicative monoid with $0_A$. There is an adjunction:
	\[
	\Hom_{\textbf{Monoids}}(M,\mathcal{F}(A)) \simeq \Hom_{\textbf{R-algebra}}(M\otimes_{\mathbb{F}_1}R,A).
	\]
	This gives rise to a functor $-\otimes_{\mathbb{F}_1}R$ from the category of monoid schemes to the category of schemes sending $X$ to $X_R:=X\otimes_{\mathbb{F}_1}R$. One can similarly construct the base-change functor for quasi-coherent sheaves on a monoid scheme $X$ to a scheme $X_R$. %Also, when $R$ is a semiring, the same construction is applied to define the base-change functor. 
	The base-change functor is constructed in the exact same way when $R$ is a semiring.  One can find the details of that construction in \cite[Section 2.3]{giansiracusa2016equations}. 
 
	\begin{myeg}
		Let $M=\{t_1^{i_1}t_2^{i_2}\cdots t_n^{i_n} \mid i_j \in \mathbb{N} \} \cup \{0\}$, the set of monomials in $n$ variables along with the absorbing element $0$. The monoid scheme $\Spec M$ is the \emph{monoid affine $n$-space} or \emph{the affine $n$-space over $\mathbb{F}_1$}, denoted by $\mathbb{A}^n$. For any field $k$, one obtains $\mathbb{A}^n\otimes_{\mathbb{F}_1}k=\mathbb{A}_k^n=\Spec k[t_1,\dots,t_n]$, the affine space over $k$. One can glue copies of $\mathbb{A}^n$ to obtain the monoid projective space $\mathbb{P}^n$ as in the classical case, or one can also apply the Proj-construction to $M$, that is $\mathbb{P}^n=\textrm{Proj}~M$. One then obtains $\mathbb{P}^n\otimes_{\mathbb{F}_1}k=\mathbb{P}^n_k$ for a field $k$. 
	\end{myeg}
	
	Let $M$ be a monoid, by an $M$-set we mean simply a nonempty pointed set with $M$-action. When the context is clear, we will interchangeably use $M$-sets and $M$-modules. For $M$-sets $A$ and $B$, one can define the tensor product $A \otimes_M B$ which satisfies the tensor-hom adjunction. Moreover, when $A$ and $B$ are monoids, $A\otimes_M B$ becomes a monoid in a natural way. If we do not assume that monoids are equipped with an absorbing element, the definition of tensor product is slightly different, see \cite{deitmar} for more details
 
	Let $X$ be a monoid scheme. By an $\mathcal{O}_X$-module, we mean a sheaf $\mathcal{F}$ of pointed sets such that $\mathcal{F}(U)$ is an $\mathcal{O}_X(U)$-set for each open subset $U$ of $X$ satisfying an obvious compatibility condition. For $\mathcal{O}_X$-modules $\mathcal{F}$ and $\mathcal{G}$, by tensor products of $M$-sets (along with sheafification), we obtain the tensor product $\mathcal{F} \otimes_{\mathcal{O}_X}\mathcal{G}$. 
	
	A special type of monoids arises from toric geometry as follows: Let $N$ be a lattice and $\sigma \subseteq N\otimes_\Z\R$ be a strongly convex rational polyhedral cone. In the construction of the affine toric variety $X_\sigma$ associated to $\sigma$, one naturally obtains a monoid $S_\sigma$, and hence the affine monoid scheme $\Spec S_\sigma$. Since our monoid $M$ should include $0_M$, we actually consider $S_\sigma \cup \{0\}$. But $S_\sigma$ and $S_\sigma \cup \{0\}$ give the same affine toric variety after the base change to a field. More generally, from a fan $\Delta$ one can construct a monoid scheme $X(\Delta)$. The fan tells us how to glue the affine pieces corresponding to each cone of the fan. Monoid schemes constructed in this way are called \emph{toric monoid schemes}. In \cite{cortinas2015toric}, the authors characterize toric monoid schemes, namely, a monoid scheme is toric if and only if it is a separated, connected, torsion free, normal monoid scheme of finite type.

	\begin{mydef}
		Let $X$ be a monoid scheme. A \emph{vector bundle} is an $\mathcal{O}_X$-module $\mathcal{F}$  such that for each $x \in X$, there exist an open neighborhood $U$ of $x$ and a set $I$ such that
		\[
		\mathcal{F}|_{U} \simeq \bigoplus_{i \in I} \mathcal{O}_X|_U.
		\]
		A vector bundle $\mathcal{F}$ is said to be \emph{rank $n$} if $|I|=n$ for each open subset $U$. 
	\end{mydef}
	
	Vector bundles on monoid schemes $X$ (only in the special case of monoid projective spaces $X=\mathbb{P}^n$) was first considered by H.-C.~Graf von Bothmer, L.~Hinsch, and U.~Stuhler in \cite{bothmer2011vector} where the authors showed that all vector bundles on $\mathbb{P}^n$ are direct sums of line bundles. Later, in \cite{pirashvili2015cohomology}, Pirashvili proved that over separated monoid schemes, all vector bundles are direct sums of line bundles.

	\subsection{Semiring schemes and tropical toric schemes}

	A semiring is a set $R$ with two binary operations (addition $+$ and multiplication $\cdot$ ) satisfying the same axioms as rings, except the existence of additive inverses. In this paper, a semiring is always assumed to be commutative unless otherwise stated. A semiring $(R,+,\cdot)$ is said to be a \emph{semifield} if $(R\backslash\{0_R\},\cdot)$ is a group. As in the case of monoids, one can mimic the classical constructions for schemes to construct analogous objects for semirings, for instance, semiring schemes and quasi-coherent sheaves. We will not repeat these constructions here, instead we refer the readers to \cite[\S 2.2]{jun2019picard} and the references therein. 
	
	\begin{mydef}\label{definition: ideals for semirings}
		Let $R$ be a semiring. 
		\begin{enumerate}
			\item 
			$R$ is said to be \emph{zero-sum free} if $a+b=0$ implies $a=b=0$ for all $a,b \in R$. 
			\item 
			An \emph{ideal} $I$ of $R$ is an additive submonoid $I \subseteq R$ such that $IR \subseteq I$. 
			\item 
			An ideal $I$ of $R$ is said to be \emph{saturated} if $x, x+y \in I$ implies $y \in I$ for all $x,y \in R$. 
			\item
			The \emph{radical} of an ideal $I$ of $R$ is
			\[
			\sqrt{I}=\{a \in R \mid a^n \in I \textrm{ for some } n \in \mathbb{N}\}.
			\]
			\item 
			The \emph{nilradical} of $R$ is
			\[
			N=\{a \in R \mid a^n=0 \textrm{  for some } n \in \mathbb{N}\}.
			\]
			\item 
			A semiring $R$ is \emph{additively idempotent} if $a+a = a$ for all $a\in R$.	
            \item A semiring $R$ is a \emph{cancellative semiring} if whenever $ab=ac$ with $a\neq 0$ implies that $b=c$ for all $a,b,c \in A$.
		\end{enumerate}
	\end{mydef}
	
	Examples of zero-sum free semirings are the additively idempotent semirings. We give the three prominent examples of idempotent semifields. 
	\begin{myeg}[Boolean semifield]\label{example: boolean}
		Consider the set $\mathbb{B}:=\{0,1\}$ with multiplication:
		\[
		0\cdot 0 = 1 \cdot 0 = 0, \quad 1\cdot 1 =1, 
		\]
		and addition:
		\[
		1+1=1, \quad 1+0=1, \quad 0+0=0.
		\]
		The set $\mathbb{B}$ with these two operations is a semifield, called the \emph{Boolean semifield}. 
	\end{myeg}
	
	\begin{myeg}[Tropical semifield]\label{example: tropical semifield}
		Consider the set $\mathbb{T}:=\mathbb{R}\cup \{-\infty\}$ with multiplication $\odot$ given by the usual addition of real numbers with $a\odot (-\infty)=-\infty$ for all $a \in \mathbb{T}$. Addition $\oplus$ is given as follows: for $a,b \in \mathbb{T}$,
		\[
		a\oplus b := \max\{a,b\}, 
		\]
		where $-\infty$ is the smallest element. Then $\mathbb{T}$ is a semifield, called the \emph{tropical semifield}. 	
	\end{myeg}
	
	\begin{myeg}
		Let $R=\mathbb{T}[x]$ be the polynomial semiring with coefficients in $\mathbb{T}$ with the polynomial addition and multiplication. The geometry of $\Spec A=\mathbb{A}^1_{\mathbb{T}}$ is quite different from the classical affine line $\mathbb{A}^1_k$ over a field $k$ mainly due to the fact that $\mathbb{T}[x]$ is not cancellative although it does not have any zero-divisors. One may consider a ``reduced model'' $\overline{\mathbb{T}[x]}$, which is cancellative, to study the geometry of $\mathbb{A}^1_\mathbb{T}$. See \cite[\S 3.4]{giansiracusa2016equations} or \cite[\S 4]{jun2018valuations}.
	\end{myeg}

 \begin{rmk} The above example points out one of several stark differences between classical algebra and tropical algebra. In tropical algebra, the lack of zero-divisors does not imply that one can perform multiplicative cancellation.  \end{rmk}

	\begin{mydef}
		An \emph{affine semiring scheme} is the prime spectrum $X=\Spec A$ equipped with a structure sheaf $\mathcal{O}_X$. A locally semiringed space is a topological space with a sheaf of semirings such that the stalk at each point has a unique maximal ideal. A \emph{semiring scheme} is a locally semiringed space which is locally isomorphic to an affine semiring scheme. 
	\end{mydef}
	
	In \cite{giansiracusa2016equations} J. Giansiracusa and N. Giansiracusa propose a special case of the semiring schemes called $\mathbb{T}$-schemes, which are locally given by an equivalence relation, which we introduce bellow.
	
	\begin{mydef}
		Let $f \in \T[x_1,\dots, x_n]$. The \emph{bend relations} of $f$ is the set of equivalences $\{ f\sim f_{\hat{i}}\}$, where $f_{\hat{i}}$ is the polynomial $f$ after removing its $i$-th monomial. For an ideal $I \subseteq \T[x_1,\dots, x_n]$ the \emph{bend congruence} of $I$ is the congruence generated by the bend relations of all $f \in I$.
	\end{mydef}

 \begin{myeg}\label{ex:bend-rel} Let $f(x, y) = 1_{\mathbb{T}}x+1_{\mathbb{T}}y+1_{\mathbb{T}} \in \mathbb{T}[x,y]$, where the $+$ represents the addition operation in this semiring. To simplify the notation we will just write $f(x, y)= x+ y + 1$. The bend relations of $f(x,y)$ is the set of equivalences $\{f(x,y) \sim x+y \sim x+1 \sim y+1\}.$
 \end{myeg}

	\begin{mydef}
		A $\mathbb{T}$-scheme is a semiring scheme that is locally isomorphic to the prime spectrum of a quotient of $\T[x_1,\dots, x_n]$ by the bend congruence of an ideal in $\T[x_1,\dots, x_n]$. In this paper we refer to these schemes as \emph{tropical schemes}.
	\end{mydef}

	We point out that in \cite{maclagan2016tropical} the term ``tropical schemes'' is reserved for $\mathbb{T}$-schemes defined by the bend relations of special ideals, called tropical ideals.

 \begin{myeg}\label{ex:scheme-trop} Let $F(x,y) = x+ y + 1\in k[x,y]$, where $k$ is a field with the trivial valuation. The tropicalization of $F(x,y)$ is obtained by taking the valuation of each coefficient and replacing the multiplication and addition on $k$ with their tropical counterparts. We will write $f(x, y) = x+y+1 \in \mathbb{T}[x,y]$ for the tropicalization of $F(x,y)$, continuing the notation from Example~\ref{ex:bend-rel}. Let $C$ be the congruence generated by the bend relations of $f(x,y)$, that is $C_f = \left<f(x,y) \sim x+y \sim x+1 \sim y+1 \right>.$ The affine scheme that $F(x,y)$ defines is $\Spec k[x,y]/\left< F(x,y)\right>$ and its scheme-theoretic tropicalization is $\Spec \mathbb{T}[x,y]/C_f$, which is a $\mathbb{T}$-scheme or a tropical scheme.
 \end{myeg}

   \begin{rmk}
       Using the idea in Example~\ref{ex:scheme-trop}, one can obtain the scheme-theoretic tropicalization of any variety. More precisely, let $X$ is an affine variety over a valued field $k$ with defining ideal $I\subseteq k[x_1, \dots, x_n]$. The scheme-theoretic tropicalization of $X$ is the prime ideal spectrum of the semiring $\mathbb{T}[x_1, \dots, x_n]/C_{\text{trop}(I)}$, where $C_{\text{trop}(I)}$ is the bend congruence of the ideal trop$(I)$. The ideal trop$(I)$ is obtained by tropicalizing each polynomial of $I$, namely,       
       taking the valuation the coefficients and replacing the multiplication and addition on $k$ with their tropical counterparts. This construction is not limited to affine varieties. The details on how to glue affine patches can be found in \cite{giansiracusa2016equations}. 

       The scheme-theoretic tropicalization carries all the information that the set-theoretic (usual) tropicalization of a variety does. More precisely, the $\mathbb{T}$-points of the tropical scheme are exactly are the points of the tropical variety. 
   \end{rmk}

	\begin{mydef}\label{def: tropicalToricSheme}
		A \emph{tropical toric scheme} is a tropical scheme, locally isomorphic to the prime spectrum of $\T[M]$, where $M$ is a monoid such that $\Spec M$ is a toric monoid scheme. Alternatively, a tropical toric scheme is obtained from a toric monoid scheme via base change.
	\end{mydef}

 \begin{myeg} Let $L = \mathbb{T}[x, y]$. Then $\Spec L$ is a tropical toric scheme, which is the tropical analogue of the affine 2-space. In this case, $L = \mathbb{T}[M]$, where $M = \mathbb{N}^2 \cup \{0\}$.  
 \end{myeg}

	\section{Free modules and projective modules over semirings} \label{section: Free modules and projective modules over semirings} 
	
	In this section, we explore several technical properties of semirings which will be used in the subsequent sections to study vector bundles. We first study properties of zero-sum free semirings as well as idempotent pairs in semirings, and connect them to topological properties of prime spectra. We also study (locally) free modules and projective modules over semirings. 
	
	\begin{mydef}\label{definition: linearly independent}
		Let $R$ be a semiring and $M$ be an $R$-module. 
		\begin{enumerate}
			\item 
			A subset $\{x_1,\dots,x_n\} \subseteq M$ is \emph{linearly independent} if
			\[
			\sum_{i=1}^na_ix_i = \sum_{i=1}^n b_ix_i
			\]
			implies that $a_i=b_i$ for all $i=1,\dots,n$. 
			\item 
			$M$ is \emph{free of rank $n$} if there exists a set $\{x_1,\dots,x_n\} $ of linearly independent generators of $M$ over $R$. 
		\end{enumerate}
	\end{mydef}
	
	One can easily see that if $M$ is a free $R$-module of rank $n$, then $M\simeq R^n$ (as an $R$-module). In fact, if $\{x_1,\dots,x_n\}$ is a set of linearly independent generators of $M$, then one has the following morphism
	\[
	\varphi:R^n \to M, \quad (a_1,\dots,a_n) \mapsto \sum_{i=1}^na_ix_i
	\]
	which is bijective.

	\begin{rmk}\label{remark: free remark}
		Linear algebra over the tropical semifield $\mathbb{T}$ (or any idempotent semifield in general) is rather subtle. If one defines the notion of linearly independent as:  $\sum_{i=1}^na_ix_i=0$ implies that $a_i=0$ for all $i$ and a free module to be one generated by this, this module does not have to be of the form $\mathbb{T}^n$ for some $n \in \mathbb{N}$. 
	\end{rmk}
	
	Over a commutative ring $A$, the rank of a free module is well defined. However, over a non-commutative ring $R$, it is possible that $R^m \simeq R^n$ but $m \neq n$. Hence, we introduce the following definition for semirings. 
	
	\begin{mydef}
		We say a semiring R has the \emph{dimension uniqueness property} (DUP) if whenever there is an isomorphism $R^m\cong R^n$ then $m=n$ for positive integers $m,n$.
	\end{mydef}
	
	\begin{lem}\label{lemma: subtractive maximal no zero-divisor}
		Let $R$ be a semiring. If $\{0\}$ is a maximal saturated ideal of $R$, then $R$ has no zero-divisors (and consequently $R$ has no nilpotents and also $\Spec R$ is connected). 
	\end{lem}
	\begin{proof}
		For each $a \neq 0 \in R$, one can observe that $\textrm{Ann}(a)=\{x \in R \mid ax=0\}$ is a saturated ideal since $ax = 0$ and $ax + ay=0$ imply $ay = ax + ay = 0$. Since $a \neq 0$, the ideal $\textrm{Ann}(a)$ is proper. Since $\{0\} \subseteq \textrm{Ann}(a)$, from the maximality of $\{0\}$, we have that $\textrm{Ann}(a)=\{0\}$. In particular, $R$ has no zero-divisors. 
	\end{proof}
	
	The following proposition ensures that the rank of a free module over a semiring is well defined when it is finite. 
	
	\begin{pro}\label{proposition: DUP proposition}
		All semirings have the DUP. 
	\end{pro}
	\begin{proof}
		Suppose that $R$ is a semiring which does not satisfy the DUP. Choose positive integers $m \neq n$ such that $R^m \simeq R^n$. If we have a homomorphism $f:R \to A$ for a semiring $A$, then $A$ does not have the DUP since otherwise we could tensor the isomorphism $R^m \simeq R^n$ with $A$ and apply the DUP in $A$ to obtain $m=n$. Let $A=R/I$, where $I$ a maximal saturated ideal. It follows that $R/I$ does not have the DUP. Therefore, without loss of generality, we may assume that the zero ideal is a maximal saturated ideal of $R$. In particular, from Lemma \ref{lemma: subtractive maximal no zero-divisor}, we may assume that $R$ has no zero-divisors. 
		
		We first consider the case when $R$ is not zero-sum free. Choose $x \neq 0$ and $y$ so that $x+y=0$.  Since there are no zero divisors, $x$ is not nilpotent, so it makes sense to localize at $x$.  The localization $R_x$ must lack the DUP due to the existence of a homomorphism $R\rightarrow R_x$. On the other hand, in $R_x$, we have 
		\[
		1 + \frac{y}{x}=0.
		\]
		Therefore $R_x$ is indeed a ring. In fact, the zero ideal is maximal among saturated ideals of $R_x$, so $R_x$ is a field.  Thus $R_x^m$ is an m-dimensional vector space over itself. The DUP property in this case is well-known, two vector spaces over the same field are isomorphic if and only if they have the same dimension. So we get a contradiction in this case.
		
		Next, suppose that $R$ is zero-sum free. Since $R$ has no zero-divisors, there is a semiring homomorphism $R \to \mathbb{B}$ sending all nonzero elements to $1$, so we get an isomorphism $\mathbb{B}^m\cong \mathbb{B}^n$.  Then we have $2^m=2^n$ so we can obtain $m=n$, and hence we also get a contradiction. 
	\end{proof}

	One can also prove a more general version of Proposition \ref{proposition: DUP proposition} allowing $n$ and $m$ to be infinite.  In fact the only place we used the finiteness assumption was for the case $R=\mathbb{B}$, which is dealt with in more generality in Example \ref{example: basis over B} below.  For completeness, we also mention that the proof of Proposition \ref{proposition: DUP proposition} shows that every semiring admits a homomorphism to a field or to $\mathbb{B}$.

	\begin{mydef}\label{definition: idempotent pair}
		Let $R$ be a semiring. 
		\begin{enumerate}
			\item 
			By an \emph{idempotent pair} of $R$, we mean a pair $(e,f)$ of elements in $R$ such that $ef=0$ and $e+f=1$. 
			\item 
			An idempotent pair $(e,f)$ is said to be nontrivial if $e,f \not \in \{0,1\}$. 
		\end{enumerate}
	\end{mydef}
	
	\begin{rmk}\label{rem: mul-idempotent}
    If $(e,f)$ is an idempotent pair, then one has $e(e+f)=e$, and hence $e^2=e$ and similarly $f^2=f$. In particular, $e$ and $f$ are multiplicatively idempotent and $e^k = e$, $f^k = f$ for all $k \in \mathbb{N}$. 
	\end{rmk}

	The intuition for working with idempotent pairs comes from direct sum decompositions. In ring theory, a decomposition of a module $M$ as a direct sum is equivalent to a choice of idempotent inside $\textrm{End}(M)$ (the idempotent serves as the projection operator onto the first summand).  This is false over semirings because split exact sequences do not give direct sum decompositions in this case.  The key issue is that in the case of rings, for any idempotent $e$, we have another idempotent $1-e$ which serves as the projection onto the second summand.  For semirings, we need to work with a pair of idempotents since we cannot recover the second projection operator from the first. 
	
	\begin{lem}Let $R$ be a semiring and $N$ be the nilradical of $R$.  Then $N$ is a saturated ideal.
	\end{lem}
	\begin{proof}
		First note that if $x+y=0$ and $x^{i-1}y^{j+1}=0$ then we obtain $x^i y^j = 0$ after multiplying $x+y=0$ by $x^{i-1}y^j$.  A simple inductive argument using this fact shows that if $x+y=0$ and $y$ is nilpotent, then $x$ is nilpotent.  Now suppose $x+y$ and $y$ are nilpotent.  Then there exists some $n$ and $p(x,y)$ such that
		\[ 0 = (x+y)^n = x^n + y p(x, y) \]
		Furthermore, $y p(x, y)$ is nilpotent, so $x^n$ (and hence $x$) is nilpotent.
	\end{proof}
	
	\begin{lem}\label{lemma: imdepotent pair lifting}
		Let $R$ be a semiring and $N$ be the nilradical of $R$. Any idempotent pair of the quotient $R/N$ can be lifted to an idempotent pair of $R$. 
	\end{lem}
	\begin{proof}
		First note that since the nilradical is saturated, reducing modulo the nilradical makes sense. Let $(e,f)$ be a pair of elements of $R$ which become an idempotent pair in $R/N$. Since $ef \in N$, we can choose $k \in \mathbb{N}$ such that $e^kf^k=0$. Let $\bar{e}$ and $\bar{f}$ denote the equivalence classes of $e$ and $f$ in $R/N$ respectively. Then by Remark~\ref{rem: mul-idempotent} and Definition~\ref{definition: idempotent pair} we have
		\[
		\bar{e}^k+\bar{f}^k = \bar{e}+\bar{f} = 1.
		\]
		If $e^k+f^k$ were a nonunit, it would be contained in a prime ideal, and we would have $0 \equiv 1$ modulo that prime.  So by contradiction, it is a unit and we may choose $u \in R$ such that $ue^k + uf^k = 1$.  We also have $u^2 e^k f^k = u^2 0 = 0$. This is enough to imply $(u e^k, uf^k)$ is an idempotent pair (that each element of the pair is idempotent follows from the other two axioms).  Furthermore, since $e^k + f^k \equiv 1$ mod $N$, it follows that $u \equiv 1$ mod $N$ and $ue^k \equiv e^k \equiv e$ mod $N$.  A similar argument holds for $f$, so our idempotent pair is congruent to the original pair mod $N$. This shows any idempotent pair in $R/N$ can be lifted to one in $R$.
	\end{proof}
	
	\begin{pro}\label{proposition: connected implies no idempotet paris}
		Let $R$ be a semiring, $\Spec R$ is connected if and only if any idempotent pair of $R$ is trivial.
	\end{pro}
	\begin{proof}
		Suppose that $\Spec R$ is connected. Denote by $V(e)$ the set of prime ideals containing $e$. Let $(e,f)$ be an idempotent pair, and consider the closed sets $V(e)$ and $V(f)$. Since $ef=0$, every prime ideal of $R$ is in either $V(e)$ or $V(f)$. Since $e+f=1$, no prime ideal can be in both $V(e)$ and $V(f)$. By connectedness, one of these sets is empty; without loss of generality let's say it's $V(e)$. Then $e$ is a unit, and so $ef=0$ implies $f=0$.  Finally $e+f=1$ implies $e=1$, showing that $(e,f)$ is trivial. 
		
		For the second assertion, suppose, on the contrary, that $\Spec R$ is disconnected. Let $N$ be the nilradical of $R$. Choose radical ideals $I$ and $J$ such that $V(I)$ and $V(J)$ form a disjoint cover of $\Spec R$ by closed sets.  No prime ideal contains both $I$ and $J$, so $I+J = R$.  Choose $e \in I$ and $f \in J$ such that $e + f = 1$.  Every prime ideal contains $I$ or $J$, so $IJ$ is a subset of $N$.  Thus $ef \equiv 0 \mod N$.  This shows that $(e, f)$ becomes an idempotent pair when reduced modulo $N$.  From Lemma \ref{lemma: imdepotent pair lifting}, we can construct an idempotent pair $(e', f')$ whose image in $R/N$ is same as the image of $(e,f)$.  Since $e$ and $e'$ are congruent mod $N$ and $J$ is radical, they are congruent mod $J$, so $e'$ is nonzero.  Similarly $f'$ is nonzero, so the idempotent pair is nontrivial, giving us a contradiction.
	\end{proof}
	
	One is tempted to say that $\Spec S$ is connected if and only if $S$ does not have any idempotent element. However, this is not true for semirings as the following example shows. 
	
	\begin{myeg}
		Let $S=\mathbb{B}[x]/\langle x^2=x \rangle$. One can easily check that there are no zero-divisors (there are 4 elements, so brute force suffices).  So $\Spec S$ is irreducible, and hence connected.  But $\bar{x} \in S$ is an idempotent.	
	\end{myeg}

	\begin{myeg}
		Let $R$ be a semiring without zero-divisors. Then $\Spec R$ is irreducible, and hence connected. It follows that $R$ has only trivial idempotent pairs.  Alternatively, one can note that for an idempotent pair $(e,f)$, the equation $ef=0$ implies $e=0$ or $f=0$ and the equation $e+f=1$ implies the other entry of the pair is $1$.
	\end{myeg}

	The following is well known for idempotent semirings. We prove the case when $R$ is a zero-sum free semiring. We note that any idempotent semiring is a zero-sum free.
	
	\begin{lem}\label{lemma: zero-sum trivial idempotent pairs}
		Let $R$ be a semiring.
		\begin{enumerate}
			\item 
			Let $S$ be a multiplicative subset of $R$. If $R$ is zero-sum free, then $S^{-1}R$ is also zero-sum free. 
			\item 
			If $R_\mathfrak{p}$ is zero-sum free for each $\mathfrak{p} \in \Spec R$, then $R$ is zero-sum free. 
		\end{enumerate}
	\end{lem}
	\begin{proof}
		For $\frac{a}{s_1},\frac{b}{s_2} \in S^{-1}R$, if $\frac{a}{s_1}+\frac{b}{s_2}=0$ then there exists $q \in S$ such that $qas_2 +qbs_1=0$. Since $R$ is zero-sum free, we have that $(qs_2)a=0$ and $(qs_1)b=0$. In particular, $\frac{a}{s_1}=\frac{b}{s_2}=0$.  
		
		For the second assertion, suppose that $a+b=0$ and $a \neq 0$. Let $I = \{x \mid ax=0\}$. Since $a \neq 0$, $I$ is a proper ideal. Let $\mathfrak{p}$ be a maximal ideal containing $I$. By the zero-sum-free property of $R_\mathfrak{p}$, we have $\frac{a}{1}=0$. Hence, there is $t \not \in \mathfrak{p}$ such that $ta=0$. This gives a contradiction since then we have $t \in I \subseteq \mathfrak{p}$. Therefore, $a=0$.  
	\end{proof}

	One may use the notion of idempotent pairs to prove that when $R$ is a zero-sum free semiring with only trivial idempotent pairs, a basis of a free module over $R$ is unique up to rescaling and permutation.  Since a basis provides a direct sum decomposition of a free module, the following lemma is the key tool.  Note that this lemma first appeared in \cite[Theorem 3.2]{izhakian2018decompositions}.
	
	\begin{lem}\label{lemma: direct summand of free module}
		Let $R$ be a zero-sum free semiring with only trivial idempotent pairs, and $M$ be a free module over $R$ with basis $S\subseteq M$.  Suppose $M=P\oplus Q$.  Then there exists a subset $S'\subseteq S$ such that $P$ is free with basis $S'$ and $Q$ is free with basis $S \setminus S'$.
	\end{lem}
	\begin{proof}We may define projection operators $e,f: M\rightarrow M$ such that $e|_P=id$, $f|_P=0$, $e|_Q=0$, and $f|_Q=id$.  It is easy to see this is an idempotent pair in $\mathop{End}(M)$.  Let $M_S(R)$ be the noncommutative semiring of matrices with rows and columns indexed by $S$ which have finitely many nonzero entries in each column.  It is straightforward to see $\mathop{End}(M) \cong M_S(R)$, so we regard $(e, f)$ as an idempotent pair in $M_S(R)$.
		
		Using $e+f=1$ and the fact that the identity matrix is diagonal, $e_{ij}+f_{ij}=0$ for $i\neq j \in S$.  By the zero-sum free property, $e$ and $f$ are diagonal.  It is then easy to check that $(e_{ii}, f_{ii})$ is an idempotent pair in $R$ for each $i\in S$.  Then either $(e_{ii}, f_{ii})=(0, 1)$ or $(e_{ii}, f_{ii})=(1, 0)$.  We let $S' = \{i\in S \mid (e_{ii}, f_{ii})=(1, 0)\}$.
		
		Let $v\in M$ and consider a basis expansion $v = \sum_{s\in S} v_s s$.  Then
		\[ ev = \sum_{s\in S}\sum_{t\in S} e_{st} v_t s = \sum_{s\in S} e_{ss} v_s s = \sum_{s\in S'} v_s s \]
		It is then easy to see $P = \mathop{Im} (e) = \mathop{span}(S')$.  Since $S'$ is a subset of $S$, it's linearly independent so the claim about $P$ follows.  The claim about $Q$ is proven similarly.
	\end{proof}
	
	\begin{pro}\label{proposition: unique basis}
		Let $R$ be a zero-sum free semiring with only trivial idempotent pairs, and $M$ be a free module over $R$. For any two bases $S,S'$ of $M$, there is a bijection $f:S\to S'$ such that $f(v)$ is a unit multiple of $v$ for each $v$ in $S$.  In particular, the basis is unique up to rescaling and permutation. 
	\end{pro}
	\begin{proof}
		We first claim that for each $v \in S$, there is a unique $w \in S'$ such that $\mathrm{Span}(v)=\mathrm{Span}(w)$.  Now, one can write $M$ as follows:
		\[
		M = \mathrm{Span}(v) \oplus \mathrm{Span}( S \backslash {v}),
		\]
		and hence $\mathrm{Span}(v)$ is spanned by some subset of $S'$ by Lemma \ref{lemma: direct summand of free module}. This subset is clearly nonempty; in particular, $\mathrm{Span}(v)$ contains some element of $S'$.  Any two elements of $\mathrm{Span}(v)$ are linearly dependent since $x(yv)=y(xv)$.  Thus $\mathrm{Span}(v)$ cannot contain more than one element of $S'$; we will use $w$ to denote the unique element of $S' \cap \mathrm{Span}(v)$.  Clearly $w$ is the only element of $S'$ such that $\mathrm{Span}(v)=\mathrm{Span}(w)$ is possible, and so we have seen that the subset of $S'$ we constructed which spans $\mathrm{Span}(v)$ is $\{w\}$.
		
		From our claim, we obtain a function $f:S \to S'$, and by reversing the roles of $S$ and $S'$, we also obtain a function $g:S' \to S$. One can easily observe that $f$ and $g$ are inverses to each other. Furthermore, since $\mathrm{Span}(v)=\mathrm{Span}(f(v))$, we conclude that $f(v)$ is a unit multiple of $v$. In particular, $S'$ can be obtained from $S$ by rescaling and permutation. 
	\end{proof}

	\begin{myeg}\label{example: basis over B}Let $M$ be a free $\mathbb{B}$-module.  Then the above proposition implies the basis of $M$ is unique.  We also give an alternative argument for this special case, which relies on showing any basis is equal to the set of atoms\footnote{An atom is an element $x\in M$ such that $0\leq x\leq y$ implies $x=0$ or $x=y$.}.  Let $X\subseteq M$ be a basis and $\mathbb{S}_{fin}(X)$ be the $\mathbb{B}$-module of finite subsets of $X$ under union.  There is a homomorphism $\psi:\mathbb{S}_{fin}(X)\rightarrow M$ sending a subset of $X$ to its sum, and the fact that $X$ is a basis implies $\psi$ is bijective.  Under this isomorphism, atoms (i.e. singleton sets) of $\mathbb{S}_{fin}(X)$ correspond to elements of $X$, establishing the claim.
	\end{myeg}

	\begin{cor}\label{corollary: invertible matrix structure}
		Let $R$ be a zero-sum free semiring with only trivial idempotent pairs.  Let $A\in \emph{GL}_n(R)$ be an invertible $n\times n$ matrix with entries in $R$.  Then $A$ has exactly one nonzero entry in each column, and this entry is a unit.
	\end{cor}
	\begin{proof}
		By applying Proposition \ref{proposition: unique basis} to the standard basis and the columns of $A$, we see that there is some permutation $f$ and a sequence of units $u_i \in R^\times$ such that the columns satisfy the following:
		\[
		A_i =u_ie_{f(i)}.
		\]
		In particular, the only nonzero entry of the $i$-th column of $A$ is in location $f(i)$.
	\end{proof}

	The following result first appeared in \cite[Proposition 2.2.2]{giansiracusa2018grassmann} when $R$ is an idempotent semiring without zero-divisors. We prove the case when $R$ is a zero-sum free semiring with only trivial idempotent pairs.

	\begin{pro}\label{proposition: exact sequence R, GL, S_n}
		Let $R$ be a zero-sum free semiring. If $R$ has only trivial idempotent pairs, then one has the following split short exact sequence of groups which is natural in $R$:
		\begin{equation}\label{eq: exact sequence R,GL, S_n}
			\begin{tikzcd}
				0 \arrow[r] &
				(R^\times)^n \arrow[r,"f"]&
				\emph{GL}_n(R) \arrow[r,"g"]
				& S_n\arrow[r]
				& 0
			\end{tikzcd}
		\end{equation}
		where $f$ is the diagonal map and $g$ sends a matrix $A$ to the unique permutation $\sigma$ such that $A_{\sigma(i)i}\neq 0$ for all $i$. 
	\end{pro}
	\begin{proof}
		Corollary \ref{corollary: invertible matrix structure} shows that the definition of $g$ makes sense.  To see $g$ is a homomorphism, let $A,B\in \emph{GL}_n(R)$.  Fix $k$ and let $i = (g(A)\circ g(B))(k)$.  One has $(AB)_{ik} = \sum_j A_{ij} B_{jk}$.  It is easy to verify that the term corresponding to $j=g(B)(k)$ is nonzero and all other terms are zero.
		
		It is easy to see that $g$ is split by the inclusion of permutation matrices.  Furthermore, the kernel of $g$ is the group of invertible diagonal matrices, which is easily seen to be $(R^\times)^n$. 
		
		Finally, one can easily check that \eqref{eq: exact sequence R,GL, S_n} is functorial in $R$. For instance, if we have a map $R \to A$, then the composition
		\[
		\textrm{GL}_n(R) \to\textrm{GL}_n(A) \to S_n
		\]
		is just the map $\textrm{GL}_n(R) \to S_n$, since the change of scalars does not effect which elements are nonzero (since the nonzero elements are units). 
	\end{proof}

	Proposition \ref{proposition: exact sequence R, GL, S_n} links $\textrm{GL}_n(R)$ with the groups of invertible diagonal matrices $(R^\times)^n$ and permutation matrices $S_n$. To be precise, we have the following.

	\begin{cor}\label{corollary: semidirect}
		Let $R$ be a zero-sum free semiring. If $R$ has only trivial idempotent pairs, then one has the following isomorphism of groups:
		\[
		(R^\times)^n \rtimes S_n \simeq \emph{GL}_n(R),
		\]
  	where $S_n$ acts on $(R^\times)^n$ by permuting factors. 
	\end{cor}

	\begin{proof}
		This directly follows from the fact that the exact sequence \eqref{eq: exact sequence R,GL, S_n} splits by considering $S_n$ as the group of permutation matrices. 
	\end{proof}
	
	In \cite{pirashvili2015cohomology}, for a monoid $M$, the following exact sequence is given: 
	\begin{equation}\label{eq: exact seqeunce M, GL, S_n}
		\begin{tikzcd}
			0 \arrow[r] &
			(M^\times)^n \arrow[r,"\varphi"]&
			\textrm{GL}_n(M) \arrow[r,"\psi"]
			& S_n\arrow[r]
			& 0
		\end{tikzcd}
	\end{equation}
	where $\varphi$ and $\psi$ are defined in the same way as \eqref{eq: exact sequence R,GL, S_n}. 
	
	Let $M$ be a cancellative monoid, $K$ be an idempotent semifield, and $R=K[M]$. By \cite[Proposition 3.2]{jun2019picard}, we have that $R^\times \simeq (K^\times \times M^\times)$. In particular, the exact sequences \eqref{eq: exact sequence R,GL, S_n} and \eqref{eq: exact seqeunce M, GL, S_n} fit into the following commutative diagram:
	\begin{equation}
		\begin{tikzcd}
			0 \arrow[r] &
			(M^\times)^n \arrow[r,"\varphi"]   \arrow[hookrightarrow,d,""] &
			\emph{GL}_n(M) \arrow[r,"\psi"] \arrow[hookrightarrow,d,""]
			& S_n\arrow[r] \arrow[d,"\emph{id}"]
			& 0 \\
			0 \arrow[r] &
			(K^\times \times M^\times)^n \arrow[r,"f"]&
			\emph{GL}_n(R) \arrow[r,"g"]
			& S_n\arrow[r]
			& 0
		\end{tikzcd}
	\end{equation}
	where the first two vertical maps are natural injections. 
	\medskip

	Now, we turn to the case of projective modules over semirings. There are two possible definitions for projective modules; (1) a direct summand of a free module, and (2) a module satisfying a homomorphism-lifting property. The first definition (direct summand) is too restrictive in the sense that projective modules are free modules in this definition as Lemma \ref{lemma: direct summand of free module} shows.
	
	To avoid the pathology in Lemma \ref{lemma: direct summand of free module}, we define projectives via homomorphism-lifting property as follows: 
	
	\begin{mydef}\label{definition: projective module}
		Let $R$ be a semiring. A module $M$ over $R$ is \emph{projective} if for any surjective homomorphism $f:N \to M$ of modules and a homomorphism $g:P \to M$, there exists a unique homomorphism $\tilde{g}:P \to N$ such that $f\circ \tilde{g} = g$.  
	\end{mydef}
	
	One can easily see that with Definition \ref{definition: projective module}, any free module is projective. We also note that in \cite{jun2019projective} and \cite{jun2020homology}, more general versions of projective modules are introduced in a framework of \emph{systems} which include modules over hyperfields. 
	
	Classically, for finitely generated modules over Noetherian rings, being projective is the same thing as being locally free. However, this is false for semirings as the following example shows. 
	
	\begin{myeg}\label{example: projective is not locally free}
		Let $\mathbb{B}$ be the Boolean semifield, and $F$ be the free module over $\mathbb{B}$ on two generators $a,b$. Let $P=F/\langle a+b=b\rangle$. For any $\mathbb{B}$-module $M$, a homomorphism $g:P \to M$ is the same as a pair $x,y$ in $M$ such that $x+y=y$.  Pick a surjective homomorphism $f:N \to M$.  Let $X,Y$ in $N$ be lifts of $x,y$, i.e. $f(X)=x$ and $f(Y)=y$. Then $X+Y$ is also a lift of $y$ and we have a homomorphism $\tilde{g}:P \to N$ sending $a$ to $X$ and $b$ to $X+Y$. Thus $P$ is projective. One can easily check that $P$ precisely has $3$ elements, and hence cannot be free since any free module $L$ over $\mathbb{B}$ should be $\mathbb{B}^n$ for some $n \in \mathbb{N}$, and hence $|L|=2^n$. Since $\mathbb{B}$ is local idempotent semiring, it follows that $P$ is not locally free. 
	\end{myeg}

	In fact, the argument in Example \ref{example: projective is not locally free} works for any idempotent semiring $R$ (rather than just $\mathbb{B}$).  We will assume we have chosen a homomorphism $R\rightarrow \mathbb{B}$, which may be constructed by sending some saturated prime to 0 and its complement to 1.
	
	Let $P_R$ be a module constructed similarly as in Example \ref{example: projective is not locally free}. Then $P_R$ is projective and for any $R$-algebra $A$, one can easily see that $P_A=P_R\otimes_R A$. Now, the special case when $A=\mathbb{B}$ shows that $P_R$ cannot be free from Example \ref{example: projective is not locally free}.  Since $R$ is arbitrary, we have for any prime $p$ that $(P_R)_p = P_R \otimes_R R_p = P_{R_p}$ is also not free, so $P_R$ is not locally free at any prime.
	
	This shows a crucial difference between projective modules over semirings and rings. For instance, classically any finitely generated projective module over a polynomial ring is free (Quillen–Suslin theorem), whereas this is no longer true as our example shows in the case of semirings.

	\section{Vector bundles for tropical schemes} \label{section: vector bundles for tropical toric schemes}

	In this section, we introduce a notion of vector bundles on a semiring scheme. We then prove two of our main theorems; (1) any vector bundle on an irreducible semiring scheme satisfying a certain local condition is a coproduct of line bundles (Theorem \ref{theorem: vector bundle decomposition}), and (2) there is a natural bijection between vector bundles on a monoid scheme $X$ satisfying some conditions and vector bundles on the semiring scheme $X_K$ (Theorem \ref{theorem: bundles stable under scalar extension}), where $X_K$ is the scalar extension of $X$ to an idempotent semifield $K$. We start with the following definition.

	\begin{mydef}
		Let $X$ be a semiring scheme. 
		\begin{enumerate}
			\item 
			A \emph{vector bundle} on $X$ is an $\mathcal{O}_X$-module $\mathcal{F}$ such that for each $x \in X$, there exists an open neighborhood $U$ of $x$ and a finite set $I$ \footnote{In this paper, we only consider the case when the rank is finite.} such that
			\[
			\mathcal{F}|_{U} \simeq \bigoplus_{i \in I} \mathcal{O}_X|_U
			\]
			as an $\mathcal{O}_X|_U$-module. A vector bundle $\mathcal{F}$ on $X$ is said to be \emph{rank $n$} if $|I|=n$ for each open subset $U$. 
			\item 
			We let $\mathbf{Vect}(X)$ be the category of vector bundles on $X$, and $\mathbf{Vect}_n(X)$ be the full subcategory consisting of vector bundles of rank $n$ on $X$ 
		\end{enumerate}
	\end{mydef}
	
	The following proposition ensures that the rank of a vector bundle on a semiring scheme is well defined. In other words, the rank of a vector bundle uniquely exists if $X$ is connected - obviously the rank of a vector bundle does not make any sense if $X$ is not connected. 
	
	\begin{pro}
		Let $X$ be a connected semiring scheme and $\mathcal{F}$ be a vector bundle on $X$. Then the rank of $\mathcal{F}$ exists and is unique. 
	\end{pro}
	\begin{proof}The uniqueness is Proposition \ref{proposition: DUP proposition}.  
		Once we have uniqueness, the existence is the same as the classical case. Hence we prove the existence. Fix $n$ and let $U$ be the set of $x \in X$ such that the stalk $\mathcal{O}_{X,x}$ at $x$ is free of rank $n$. For any $x \in X$, we can pick a neighborhood $V_x$ of $x$ on which $\mathcal{F}$ is trivial. If $x \in U$, then $\mathcal{F}$ is free of rank $n$ on $V_x$ (if $\mathcal{F}$ had any other rank on $V_x$, $\mathcal{F}$ would have that rank when localizing at $x$ contradicting Proposition \ref{proposition: DUP proposition}). Hence it is free of rank $n$ when localized at any point of $V_x$ so $V_x\subseteq U$ once $x \in U$. By connectedness of $X$ either $U =\emptyset$ or $U=X$. Finally, by picking $n$ to be the rank of some stalk, we can land in the case $U=X$, and hence $\mathcal{F}$ has rank $n$.
	\end{proof}

	Let $X$ be a semiring scheme. The coproduct and tensor product of modules gives rise to the following operations on vector bundles:
	\begin{equation}\label{eq: coproduct}
		\coprod: \mathbf{Vect}_m(X) \times \mathbf{Vect}_n(X) \to \mathbf{Vect}_{m+n}(X)
	\end{equation}
	\begin{equation}\label{eq: tensor}
		\otimes_{\mathcal{O}_X}:\mathbf{Vect}_m(X) \times \mathbf{Vect}_n(X) \to \mathbf{Vect}_{mn}(X)
	\end{equation}
	Let $Vect_n(X)$ be the set of isomorphism classes of vector bundles of rank $n$ on $X$. Then clearly the maps \eqref{eq: coproduct} and \eqref{eq: tensor} define corresponding maps for $Vect_n(X)$.

	\begin{mythm} \label{theorem: bundle cohomology bijections}
		Let $X$ be a semiring scheme (or more generally a semiringed space).  There is a natural bijection
		\[
		Vect_n(X) \simeq H^1(X,\emph{GL}_n(\mathcal{O}_X)).
		\]
		Moreover one has a commutative diagram
		\begin{equation}\label{eq: diagram for classifying sums of vector bundles}
			\begin{tikzcd}
				Vect_m(X) \times Vect_n(X) \arrow{d}{\simeq} \arrow{r}{}
				& Vect_{m+n}(X) \arrow{d}{\simeq} \\
				H^1(X,\emph{GL}_m(\mathcal{O}_X)) \times H^1(X,\emph{GL}_n(\mathcal{O}_X))\arrow{r}{}
				& H^1(X,\emph{GL}_{m+n}(\mathcal{O}_X))
			\end{tikzcd}
		\end{equation}
		
		and a diagram
		\begin{equation}\label{eq: diagram for classifying sums of line bundles}
			\begin{tikzcd}
				\Pic(X)^n \arrow{d}{\simeq} \arrow{r}{}
				& Vect_{n}(X) \arrow{d}{\simeq} \\
				H^1(X,(\mathcal{O}_X^\times)^n) \arrow{r}{}
				& H^1(X,\emph{GL}_n(\mathcal{O}_X))
			\end{tikzcd}
		\end{equation}
	\end{mythm}
	\begin{proof}The proof is the same as the classical case of schemes, and is quite standard.  Fix an open cover $\mathcal{U}$ of $X$. It suffices to show that 
		\[
		Vect_n(\mathcal{U}) \simeq \check{H}^1(\mathcal{U},\emph{GL}_n(\mathcal{O}_X)),
		\]
		where $Vect_n(\mathcal{U})$ denotes the set of vector bundles trivialized by $\mathcal{U}$ - one then obtains the first part of the theorem by taking colimits.
		
		Given some $\mathcal{F} \in Vect_n(\mathcal{U})$, for each $U_i \in \mathcal{U}$ one has isomorphisms $\psi_i: \mathcal{F}(U_i)\rightarrow \mathcal{O}_X(U_i)^n$.  For each pair $U_i, U_j \in \mathcal{U}$, define 
		\begin{equation}\label{eq: 1-cocycle construction}
			\theta_{ij} = \psi_i \mid_{U_i\cap U_j} \psi_j^{-1} \mid_{U_i\cap U_j} \in \emph{GL}_n(\mathcal{O}_X) (U_i \cap U_j).
		\end{equation}
		It is routine to check that $\theta$ is a 1-cocycle.  If we made a different choice of trivialization, say $\phi_i \psi_i$ with $\phi_i \in \emph{GL}_n(\mathcal{O}_X(U_i))$, we would get the cocycle $\phi_i\theta_{ij}\phi_j^{-1}$ which defines the same cohomology class.  Conversely, given a 1-cocycle $\theta$, we can construct a vector bundle $\mathcal{F}$ by defining a section $s\in \mathcal{F}(V)$ to be a family of sections $s_i \in \mathcal{O}_X(V\cap U_i)^n$ for each $U_i\in \mathcal{U}$ satisfying
		\[
		s_i = \theta_{ij}\mid_{V\cap U_i \cap U_j} s_j
		\]
		for each $U_i, U_j\in \mathcal{U}$.
		
		Fix a rank $m$ bundle $\mathcal{A}$ and a rank $n$ bundle $\mathcal{B}$.  Let $\mathcal{U}$ be an open cover on which both bundles are trivial.  We have isomorphisms $\psi_{\mathcal{A},i}: \mathcal{A}(U_i)\rightarrow \mathcal{O}_X(U_i)^m$  and $\psi_{\mathcal{B},i}: \mathcal{B}(U_i)\rightarrow \mathcal{O}_X(U_i)^n$.  We may define isomorphisms $\psi_i: (\mathcal{A} \oplus \mathcal{B})(U_i)\rightarrow \mathcal{O}_X(U_i)^{m+n}$ by 
		\[
		\psi_i(s_\mathcal{A}, s_\mathcal{B}) = (\psi_{\mathcal{A},i}(s_\mathcal{A}), \psi_{\mathcal{B},i}(s_\mathcal{B}))
		\]
		for any $(s_\mathcal{A}, s_\mathcal{B}) \in (\mathcal{A} \oplus \mathcal{B})(U_i)$.  We may then construct $\theta$ via equation \eqref{eq: 1-cocycle construction}.  The corresponding cohomology class is the one obtained by mapping $(\mathcal{A}, \mathcal{B})$ via the top and right arrows of \eqref{eq: diagram for classifying sums of vector bundles}.  By the structure of $\psi_i$, the matrices $\theta_{ij}$ are block diagonal with blocks of the form $\psi_{\mathcal{A}, i}\psi_{\mathcal{A}, j}^{-1}$ and $\psi_{\mathcal{B}, i}\psi_{\mathcal{B}, j}^{-1}$.  This block diagonal description is precisely what we get by following the left and bottom arrows of \eqref{eq: diagram for classifying sums of vector bundles}.
		
		Diagram \eqref{eq: diagram for classifying sums of line bundles} can be proven either by repeated application of diagram \eqref{eq: diagram for classifying sums of vector bundles} or via the same technique.
	\end{proof}

	\begin{lem}\label{lemma: cohomology vanishing for constant sheaf}
		Let $X$ be an irreducible semiring scheme, $G$ be a group, and $\underline{G}$ be the constant sheaf on $X$ associated to $G$. Then
		\[
		H^1(X,\underline{G})=0.
		\]
	\end{lem}
	\begin{proof}
		Since $X$ is irreducible, every open subset of $X$ is connected. In particular, $\underline{G}$ is a sheaf of constant functions. A simple calculation will then show the cohomology vanishes.	
	\end{proof}

	\begin{rmk}
		If $G$ is an abelian group in Lemma \ref{lemma: cohomology vanishing for constant sheaf}, one has that $H^p(X,\underline{G})=0$ for all $p>0$. For instance, see  \cite[\href{https://stacks.math.columbia.edu/tag/02UW}{Lemma 02UW}]{stacks-project}.
	\end{rmk}

	One can sheafify Proposition \ref{proposition: exact sequence R, GL, S_n} as follows. 
	
	\begin{pro}\label{proposition: sheafified exact sequence}
		Let $X$ be an irreducible semiring scheme which is locally isomorphic to $\Spec R$, where $R$ is a zero-sum free semiring. Let $\underline{S_n}$ be the constant sheaf associated to $S_n$. Then, one has the following split exact sequence of sheaves:
		\begin{equation}\label{eq: sheaf exact sequence R,GL, S_n}
			\begin{tikzcd}
				0 \arrow[r] &
				(\mathcal{O}_X^\times)^n \arrow[r,"f"]&
				\emph{GL}_n(\mathcal{O}_X) \arrow[r,"g"]
				& \underline{S_n}\arrow[r]
				& 0
			\end{tikzcd}
		\end{equation}
	\end{pro}
	\begin{proof}
		Let $U=\Spec R$ be an affine open subset of $X$. Since $X$ is irreducible, $U$ is connected and hence $R$ has only trivial idempotent pairs by Proposition \ref{proposition: connected implies no idempotet paris}. Furthermore, $R$ is zero-sum free from Lemma \ref{lemma: zero-sum trivial idempotent pairs} since from our assumption on $X$, $\mathcal{O}_{X,\mathfrak{p}}$ is zero-sum free for each $\mathfrak{p} \in X$. Furthermore, the map $g$ in \eqref{eq: exact sequence R,GL, S_n} (in Proposition \ref{proposition: exact sequence R, GL, S_n}) is compatible with localization. Now, \eqref{eq: sheaf exact sequence R,GL, S_n} is obtained by gluing \eqref{eq: exact sequence R,GL, S_n} for each affine open subset $U$ of $X$. One can easily check that $g$ is split since for each affine open subset of $X$, $g$ is split from Proposition \ref{proposition: exact sequence R, GL, S_n}, and the splitting is compatible with gluing. 
	\end{proof}

	\begin{mythm}\label{theorem: vector bundle decomposition}
		Let $X$ be an irreducible semiring scheme which is locally isomorphic to $\Spec R$, where $R$ is a zero-sum free semiring with only trivial idempotent pairs. Then any vector bundle of rank $n$ on $X$ is a coproduct of $n$ copies of line bundles on $X$. Moreover, this decomposition is unique up to permuting summands. 
	\end{mythm}
	\begin{proof}Consider the long exact sequence associated to \eqref{eq: sheaf exact sequence R,GL, S_n}:
		\[
		H^1(X, (\mathcal{O}_X^\times)^n)) \xrightarrow{f_*}
		H^1(X, \emph{GL}_n(\mathcal{O}_X)) \xrightarrow{g_*} 
		H^1(X, \underline{S_n})
		\]
		Note that by diagram \eqref{eq: diagram for classifying sums of line bundles}, $f_*$ is essentially the direct sum map $\Pic(X)^n\rightarrow Vect_n(X)$.
		
		By Lemma \ref{lemma: cohomology vanishing for constant sheaf}, $f_*$ is surjective, so every vector bundle is a sum of line bundles.  By \cite[Proposition 5.3.1]{grothendieck1958general}, two elements of $H^1(X, (\mathcal{O}_X^\times)^n)$ map to the same element of $H^1(X, \emph{GL}_n(\mathcal{O}_X))$ if and only if they are in the same orbit of the action of $S_n$ on $H^1(X, (\mathcal{O}_X^\times)^n)$.  This implies that this decomposition is unique up to permuting summands. 
	\end{proof}
	
	\begin{cor}\label{corollary: vector bundle on affine is trivial}
		Let $M$ be a cancellative monoid and $K$ be an idempotent semifield. Then any vector bundle on $X=\Spec K[M]$ is trivial. 
	\end{cor}
	\begin{proof}
		From \cite[Theorem 3.12]{jun2019picard}, we know that $\Pic(X)$ is trivial. Since $R=K[M]$ satisfies the condition in Theorem \ref{theorem: vector bundle decomposition}, any vector bundle is a product of trivial line bundles, and hence trivial. 
	\end{proof}
	
	\begin{myeg}\label{example: trivial vector bundle}
		Let $M=\mathbb{F}_1\langle x_1,...,x_n \rangle$ and $K=\mathbb{T}$. It follows from Corollary \ref{corollary: vector bundle on affine is trivial} that any vector bundle on the affine space $\mathbb{A}^n_{\mathbb{T}}=\Spec \mathbb{T}[x_1,...,x_n]$ over the tropical semifield $\mathbb{T}$ is trivial. 
	\end{myeg}

	Now, we prove that there is a natural bijection between the set $Vect_n(X)$ of isomorphism classes of vector bundles on a monoid scheme $X$ and the set $Vect_n(X_K)$ of isomorphism classes of vector bundles on a semiring scheme $X_K$, where $X_K$ is the scalar extension of $X$ to an idempotent semifield $K$. To this end, we require that $X$ satisfies the following technical condition.  
	
	\begin{condition}\label{condition: condition on open cover}
		Let $X$ be an irreducible monoid scheme. Suppose that $X$ has an open affine cover $\mathcal{U}=\{U_\alpha\}$ such that any finite intersection of the sets $U_\alpha$ is isomorphic to the prime spectrum of a cancellative monoid. 
	\end{condition}

	\begin{mythm}\label{theorem: bundles stable under scalar extension}
		Let $X$ be an irreducible monoid scheme satisfying Condition \ref{condition: condition on open cover}, and $K$ be an idempotent semifield. Then, there exists a natural bijection between $Vect_n(X)$ and $Vect_n(X_K)$. 
	\end{mythm}
	\begin{proof}
		We first claim that $X_K$ is irreducible. In fact, since $X$ is irreducible, $X$ has an affine open cover $\mathcal{U}=\{U_i\}_{i \in I}$ such that $U_i$ is irreducible and $U_i \cap U_j \neq \emptyset$ for all $i,j \in I$. One can easily check that $U_i \times_{\Spec \mathbb{F}_1} K$ is irreducible and $\{U_i \times_{\Spec \mathbb{F}_1} K\}_i$ is an affine open cover (see, \cite[Lemma 3.1]{jun2019picard}). Furthermore, $(U_i \times_{\Spec \mathbb{F}_1} K) \cap (U_j \times_{\Spec \mathbb{F}_1} K) \neq \emptyset$ for all $i,j,$. This shows that $X_K$ is irreducible. 
		
		Now, from \cite[Proposition 3.1]{pirashvili2015cohomology}, we have a bijection $Vect_n(X) \cong (\Pic(X)^n)_{S_n}$, from the set of vector bundles to the set of $S_n$-orbits of $n$-tuples of line bundles.  From \cite[Theorem 3.19]{jun2019picard}, we have a bijection $\Pic(X)\cong \Pic(X_K)$, which induces a bijection $(\Pic(X)^n)_{S_n} \cong (\Pic(X_K)^n)_{S_n}$.  Theorem \ref{theorem: vector bundle decomposition} gives a bijection $(\Pic(X_K)^n)_{S_n} \cong Vect_n(X_K)$.  By composition, we get a bijection $Vect_n(X) \cong Vect_n(X_K)$.
	\end{proof}
	
	In \cite{jun2019picard}, it is shown that if $X$ is an irreducible monoid scheme satisfying Condition \ref{condition: condition on open cover}, then one has the following natural isomorphisms:
	\[
	\Pic(X) \simeq \Pic(X_K) \simeq \Pic(X_k),
	\]
	where $K$ is an idempotent semifield, and $k$ is a field. The isomorphism $\Pic(X)\simeq \Pic(X_k)$ was proved by Flores and Weibel in \cite{flores2014picard}, and we used that result to obtain these isomorphisms. One may ask whether or not a similar result is true for vector bundles, but obviously one cannot have the above correspondence with vector bundles. So, the next question is, with the same notation as Theorem \ref{theorem: bundle cohomology bijections}, whether or not one has the following commutative diagram:
	\begin{equation} \label{eq: commutative diagram vector bundles with a field}
		\begin{tikzcd}
			Vect_n(X_K) \times Vect_m(X_K) \arrow{d}{} \arrow{r}{}
			& Vect_{n+m}(X_K) \arrow{d}{} \\
			Vect_n(X_k) \times Vect_m(X_k)\arrow{r}{}
			& Vect_{n+m}(X_k)
		\end{tikzcd}
	\end{equation}
	But, it turns out that even this is false. There is some linear relation (inside $K_0$ for quasi-projective varieties over a field $k$) between line bundles, i.e. we have non-negative integers $a_i, b_i$ and line bundles $L_i$ such that $\sum a_i L_i = \sum b_i L_i$ in $K_0$. After adding enough to both sides, we get a nontrivial relation that also holds before taking the group completion (think of this as analogous to how $\frac{a}{s}=\frac{b}{t}$ in a localization does not mean $ta=bs$, but means that multiplying both sides by some element of $S$ makes this true).  This result implies that if decompositions into line bundles exist (so that whatever we added to both sides is a sum of line bundles), then such decompositions are not unique. In particular, \eqref{eq: commutative diagram vector bundles with a field} is not well-defined.

	\section{Topological $\mathbb{T}$-vector bundles} \label{section: topological bundles}
	
	In this section, we introduce the notion of topological $\mathbb{T}$-vector bundles on a topological space $X$ equipped with the structure sheaf $\mathcal{O}_X$ of continuous $\mathbb{T}$-valued functions. The main result in this section states that for a connected paracompact Hausdorff space $X$, there is a canonical split surjection from the set of isomorphism classes of topological $\mathbb{T}$-vector bundles of rank $n$ on $X$ to the set of isomorphism classes of $n$-fold covering spaces of $X$ (Theorem \ref{theorem: main theorem of topological bundles}).
	
	Our motivation arises from the two facts: $(1)$ the scheme-theoretic tropicalization realizes the set-theoretic tropicalization as the set of $\mathbb{T}$-rational points, and $(2)$ a priori, the set of $\mathbb{T}$-rational points is merely a set, however, one may impose a topology which is induced from the Euclidean topology of $\mathbb{T}$. Therefore, one may expect certain relationships between vector bundles on tropical schemes and topological $\mathbb{T}$-vector bundles on their sets of $\mathbb{T}$-rational points. We first recall how one can endow a topology on sets of rational points. 
	
	Let $X$ be a scheme and $k$ be a topological field. There is a canonical way to impose a topology on $X(k)$, called the fine Zariski topology. In tropical geometry, the fine Zariski topology has been used to give a homeomorphism between Berkovich analytification and a set of rational points of a scheme over some ``generalized algebraic structures'', for instance, see \cite{giansiracusa2014universal}, \cite{lorscheid2015scheme}, and \cite{jun2017geometry}. 
	
	Let's briefly recall the definition of the fine Zariski topology in perspective of $\mathbb{T}$. Let $X$ be a semiring scheme and $\mathbb{T}$ be equipped with the Euclidean topology. First, we consider when $X$ is an affine, i.e., $X=\Spec A$ for some semiring $A$. In this case, we have
	\[
	X(\mathbb{T})=\Hom(\Spec \mathbb{T}, X)=\Hom (A,\mathbb{T}), 
	\] 
	and hence we have the following:
	\begin{equation}
		X(\mathbb{T})=\Hom(A,\mathbb{T}) \subseteq \prod_{a \in A}\mathbb{T}^{(a)}. 
	\end{equation}
	We impose the product topology on $\prod_{a \in A}\mathbb{T}^{(a)}$ and then impose the subspace topology on $X(\mathbb{T})$. This is called the \emph{affine topology}. In other words, the affine topology is the weakest topology on $\Hom (A,\mathbb{T})$ such that for each $a \in A$, the evaluation map 
	\[
	ev_a:\Hom(A,\mathbb{T}) \to \mathbb{T},\quad f \mapsto f(a) 
	\]
	is continuous. In this case, one can easily check that this topology is functorial in both $A$ and $\mathbb{T}$, that is when $\mathbb{T}$ is replaced with other topological semifields.
	
	Next, consider the case when $X$ is a semiring scheme. The \emph{fine Zariski topology} on $X(\mathbb{T})$ is the finest topology such that any morphism $f_Y:Y \to X$ from an affine semiring scheme $Y$ to $X$ induces the following continuous map 
	\[
	f_Y(\mathbb{T}):Y(\mathbb{T}) \to X(\mathbb{T}),
	\] 
	where $Y(\mathbb{T})$ is equipped with the affine topology. With this, one has the following (see \cite{lorscheid2015scheme}):
	
	\begin{enumerate}
		\item 
		If $X$ is an affine semiring scheme, then the affine topology and the fine topology agree on $X(\mathbb{T})$. 
		\item 
		Let $f:Y \to X$ be a morphism of semiring schemes. Then the induced map,
		\[
		f(\mathbb{T}):Y(\mathbb{T}) \to X(\mathbb{T})
		\]
		is continuous, where $Y(\mathbb{T})$ and $X(\mathbb{T})$ are equipped with the fine Zariski topology. 
		\item 
		If $\{U_i\}$ is an affine open covering of a semiring scheme $X$. Then $\{U_i(\mathbb{T})\}$ is an open covering of $X(\mathbb{T})$. 
		\item 
		The canonical map $p:X(\mathbb{T}) \to X$ is continuous, in other words, the fine Zariski topology is finer than ordinary Zariski topology.
	\end{enumerate}
	\medskip

	Let $X$ be a topological space, by a \emph{topological $\mathbb{T}$-vector bundle}, we mean a topological space $E$ with a continuous map $\pi:E \to X$ such that for each $x \in X$, there exists an open neighborhood $U_x$ of $x$ and $\pi^{-1}(U_x)$ is homeomorphic to $U_x \times \mathbb{T}^n$ for some $n \in \mathbb{N}$. We also require that transition maps to be $\mathbb{T}$-linear.

	\begin{pro}\label{proposition: T vector bundles are locally free}
		Let $X$ be a topological space, and $\mathcal{O}_X$ be the sheaf of continuous $\mathbb{T}$-valued functions.
  %\footnote{This is a sheaf of semirings due to the semiring structure of $\mathbb{T}$ - Kalina: I removed this since I am not sure what it actually says}. 
  There is an equivalence of categories between topological $\mathbb{T}$-vector bundles on $X$ and locally free $\mathcal{O}_X$-modules, which sends a vector bundle to its sheaf of continuous sections.
	\end{pro}
	\begin{proof}
		Let $\mathbf{Vect}(X)$ be the category of topological $\mathbb{T}$-vector bundles on $X$ and $\mathbf{S}(X)$ be the category of locally free $\mathcal{O}_X$-modules on $X$. For each $\pi:E \to X$ in $\mathbf{Vect}(X)$, as in the classical case, we can define a presheaf $\mathcal{E}$; for each open subset $U$ of $X$, $\mathcal{E}(U)$ is the set of continuous sections on $U$. In fact, one can easily see that $\mathcal{E}(U)$ is an $\mathcal{O}_X(U)$-module and the module structure is compatible with restriction maps, and hence $\mathcal{E}$ is indeed a sheaf. This construction is functorial, so we have a functor $F:\mathbf{Vect}(X) \to \mathbf{S}(X)$ sending $E$ to $\mathcal{E}$. 
		
		Conversely, if $\mathcal{E}$ is a locally free sheaf over $\mathcal{O}_X$, then for each $x \in X$, we have an open neighborhood $U_x$ such that $\mathcal{E}(U) \simeq \mathcal{O}_X(U)^n$ for some $n$. Since $\mathcal{O}_X(U)$ is the semiring of continuous functions from $U$ to $\mathbb{T}$, we can see $\mathcal{O}_X(U)^n$ is nothing but the set of continuous sections of a trivial topological $\mathbb{T}$-vector bundle $E_U=U \times \mathbb{T}^n \to U$. Now, one may observe that $E_U$ are glued to obtain a topological $\mathbb{T}$-vector bundle $\pi:E \to X$. As this construction is functorial, this defines a functor $G:\mathbf{S}(X) \to \mathbf{Vect}(X)$. Finally, it is clear that $G$ is a quasi-inverse of $F$, showing the equivalence of categories. 
	\end{proof}

	\begin{lem}\label{lemma: topological bundle lemma}
		Let $X$ be a topological space, and $C(X,\mathbb{T})$ be the set of continuous $\mathbb{T}$-valued functions. Then continuous $\textrm{GL}_n(\mathbb{T})$-valued functions on $X$ are the same as invertible matrices with entries in $C(X,\mathbb{T})$. 
	\end{lem}
	\begin{proof}
		Let $f:X \to \textrm{GL}_n(\mathbb{T})$ be a continuous function. We first claim that the automorphism $S: \textrm{GL}_n(\mathbb{T}) \to \textrm{GL}_n(\mathbb{T})$ sending $M$ to $M^{-1}$ is continuous. In fact, for $\sigma \in S_n$, there is an open set in $\textrm{GL}_n(\mathbb{T})$ such that $a_{i\sigma(i)}\neq 0$ for each $i$, and one can check continuity on these open sets. We also have continuous projections $\pi_{ij}:\textrm{GL}_n(\mathbb{T}) \to \mathbb{T}$ sending $(a_{ij})$ to $a_{ij}$. By composing the projections $\pi_{ij}$ and $f$, we obtain continuous functions $f_{ij}:X \to \mathbb{T}$, and hence we obtain a matrix $M_f=(f_{ij}) \in \textrm{M}_n(C(X,\mathbb{T}))$. But, one can easily see that $M_{Sf}$ is the inverse of $M_f$ in $\textrm{M}_n(C(X,\mathbb{T}))$. In particular, $M_f \in \textrm{GL}_n(C(X,\mathbb{T}))$. This provides the desired correspondence.  Constructing a continuous function with values in $\textrm{GL}_n(\mathbb{T})$ from an element of $\textrm{GL}_n(C(X,\mathbb{T}))$ can be done similarly, and the constructions are clearly inverse to each other.
	\end{proof}
	
	\begin{lem}\label{lemma: topological bundle lemma2}
		Let $X$ be a topological space. Consider $C(X,\mathbb{T})$ as a semiring with addition and multiplication induced by those of $\mathbb{T}$. Then $X$ is connected if and only if any idempotent pair of $C(X,\mathbb{T})$ is trivial.
	\end{lem}
	\begin{proof}
		Suppose that $(e,f)$ is an idempotent pair of $C(X,\mathbb{T})$. Then, one can easily see that $(e(x),f(x))$ is an idempotent pair of $\mathbb{T}$ for each $x \in X$. In particular, $e$ and $f$ are continuous $\{0,1\}$-valued functions, and hence they should be constant functions by connectedness of $X$. Now, one can see that $e,f \in \{0,1\}$. The converse is clear since on a disconnected space $X$, one can build a complementary pair of continuous $\{0,1\}$-valued functions giving us a nontrivial idempotent pair. 
	\end{proof}
	
	The following is an analogue of Proposition \ref{proposition: sheafified exact sequence} for topological $\mathbb{T}$-vector bundles. 
	
	\begin{pro}\label{proposition: topological bundle proposition}
		Let $X$ be a locally connected topological space. Then, we have the following split exact sequence:
		\[
		\begin{tikzcd}
			0 \arrow[r] &
			(\mathcal{O}_X^\times)^n \arrow[r,"f"]&
			\emph{GL}_n(\mathcal{O}_X) \arrow[r,"g"]
			& \underline{S_n}\arrow[r]
			& 0
		\end{tikzcd}
		\]
	\end{pro}
	\begin{proof}
		We first note that for an open subset $U$ of $X$, we have $(\mathcal{O}_X^\times(U))^n=C(U,(\mathbb{T}^\times)^n)$, $\emph{GL}_n(\mathcal{O}_X(U))=C(U,\emph{GL}_n(\mathbb{T}))$ by Lemma \ref{lemma: topological bundle lemma}, and $\underline{S_n}=C(U,S_n)$. From Proposition \ref{proposition: exact sequence R, GL, S_n}, we have the following split exact sequence:
		\begin{equation}\label{eq: exseq}
			\begin{tikzcd}
				0 \arrow[r] &
				(\mathbb{T}^\times)^n \arrow[r,"\varphi"]&
				\emph{GL}_n(\mathbb{T}) \arrow[r,"\psi"]
				& S_n\arrow[r]
				& 0
			\end{tikzcd}
		\end{equation}
		Now, from \eqref{eq: exseq}, we have the following induced sequence of maps. Furthermore the map $\psi_*$ splits.
		\begin{equation}\label{eq: sheaf ex1}
			\begin{tikzcd}
				0 \arrow[r] &
				C(U,(\mathbb{T}^\times)^n) \arrow[r,"\varphi_*"]&
				C(U,\emph{GL}_n(\mathbb{T})) \arrow[r,"\psi_*"]
				& C(U,S_n)\arrow[r]
				& 0
			\end{tikzcd}
		\end{equation}
		where $C(U,Y)$ denoted the group of continuous functions from $U$ to $Y$. Moreover the maps will be natural in $U$ and therefore compatible with restriction maps. In particular, we have the following sequence of sheaves:
		\begin{equation}\label{eq: sheaf ex}
			\begin{tikzcd}
				0 \arrow[r] &
				(\mathcal{O}_X^\times)^n \arrow[r,"f"]&
				\emph{GL}_n(\mathcal{O}_X) \arrow[r,"g"]
				& \underline{S_n}\arrow[r]
				& 0
			\end{tikzcd}
		\end{equation}
		Now we need to check the sequence of sheaves \eqref{eq: sheaf ex} is exact. It is enough to show that we have an exact sequence of sections over each connected open subset of $X$ since $X$ is locally connected.
		Let $U$ be a connected open subset of $X$. Clearly the semiring $R_U:=C(U,\mathbb{T})$ is zero-sum free, and from Lemma \ref{lemma: topological bundle lemma2} $R_U$ has only trivial idempotent pairs. In particular, we can apply Proposition \ref{proposition: exact sequence R, GL, S_n} to obtain the following exact sequence:
		\[
		\begin{tikzcd}
			0 \arrow[r] &
			C(U,(\mathbb{T}^\times)^n) \arrow[r,"\varphi_*|_U"]&
			C(U,\emph{GL}_n(\mathbb{T})) \arrow[r,"\psi_*|_U"]
			& C(U,S_n)\arrow[r]
			& 0
		\end{tikzcd}
		\]
		This shows that the sequence of sheaves \eqref{eq: sheaf ex} is exact. 
	\end{proof}

	\begin{mythm}\label{theorem: main theorem of topological bundles}
		Let $X$ be a locally connected paracompact Hausdorff space.  Then there is a canonical split surjection from the set of isomorphism classes of topological $\mathbb{T}$-vector bundles of rank $n$ on $X$ to the set of isomorphism classes of $n$-fold covering spaces of $X$.  Furthermore, a topological $\mathbb{T}$-vector bundle is trivial if and only if its associated covering space is trivial.
	\end{mythm}
	\begin{proof}
		By Proposition \ref{proposition: T vector bundles are locally free} and Theorem \ref{theorem: bundle cohomology bijections}, vector bundles of rank $n$ are classified by $H^1(X, \emph{GL}_n(\mathcal{O}_X))$.  It is well-known that $n$-fold covering spaces are the same as fiber bundles whose fiber is a discrete space of cardinality $n$, and the proof of Theorem \ref{theorem: bundle cohomology bijections} shows these can be classified by $H^1(X, \underline{S_n})$.  Thus we must construct a map $H^1(X, \emph{GL}_n(\mathcal{O}_X)) \rightarrow H^1(X, \underline{S_n})$ with the stated properties.
		
		Proposition \ref{proposition: topological bundle proposition} gives us a split exact sequence
		\[
		\begin{tikzcd}
			0 \arrow[r] &
			(\mathcal{O}_X^\times)^n \arrow[r,"f"]&
			\emph{GL}_n(\mathcal{O}_X) \arrow[r,"g"]
			& \underline{S_n}\arrow[r]
			& 0
		\end{tikzcd}
		\]
		Taking cohomology gives the following exact sequence of pointed sets, in which the last arrow is a split surjection.
		\[
		\begin{tikzcd}
			H^1(X, (\mathcal{O}_X^\times)^n) \arrow[r,"f_*"]&
			H^1(X, \emph{GL}_n(\mathcal{O}_X)) \arrow[r,"g_*"]
			& H^1(X, \underline{S_n})
		\end{tikzcd}
		\]
        More precisely, we have a sequence of functions such that the image of $f_*$ is equal to the inverse image of the base point of $H^1(X,\underline{S_n})$ by $g_*$ which we call the kernel of $g_*$.
		Since $\mathbb{T}^\times = \mathbb{R}$, $\mathcal{O}_X^\times$ is the sheaf of continuous $\mathbb{R}$-valued functions under usual addition of real numbers. By paracompactness, there is a continuous partition of unity, which makes $\mathcal{O}_X^\times$ acyclic. It follows that 
		\[
		H^1(X, (\mathcal{O}_X^\times)^n) = 0.
		\]
		We have constructed a split surjection $H^1(X, \emph{GL}_n(\mathcal{O}_X)) \rightarrow H^1(X, \underline{S_n})$ of pointed sets with trivial kernel.  Identifying $H^1(X, \emph{GL}_n(\mathcal{O}_X))$ with the pointed set of isomorphism classes of rank $n$ vector bundles and $H^1(X, \underline{S_n})$ with the pointed set of isomorphism classes of $n$-fold covering spaces completes the proof.
	\end{proof}

	\begin{cor}\label{corollary: 1}
		Let $X$ be a locally connected paracompact Hausdorff space.  Then all topological $\mathbb{T}$-line bundles on $X$ are trivial.
	\end{cor}
	\begin{proof}This follows from the fact that any $1$-fold covering space is trivial.
	\end{proof}
	
	We note that L.~Allermann \cite[Corollary 1.25]{allermann2012chern} also proved a similar result as in Corollary \ref{corollary: 1} in the context of tropical cycles. 
	
	\begin{cor}\label{corollary: 2}
		Let $X$ be a paracompact Hausdorff space which is locally path-connected and semilocally simply connected.  If $X$ is simply connected, then all topological $\mathbb{T}$-vector bundles on $X$ are trivial.
	\end{cor}
	\begin{proof}Under the stated conditions, $n$-fold covering spaces correspond to sets of size $n$ together with an action of $\pi_1(X)$.  In the simply connected case, this action (and hence the covering space) must be trivial.
	\end{proof}

	\begin{rmk}
		Let $X$ be an algebraic variety and $\textrm{Trop}(X)$ be its scheme-theoretic tropicalization. From Corollaries \ref{corollary: 1} and \ref{corollary: 2}, one may determine when all topological $\mathbb{T}$-vector bundles on $\textrm{Trop}(X)(\mathbb{T})$ are trivial. For instance, when $X$ is a projective variety, all topological $\mathbb{T}$-line bundles on $
		\textrm{Trop}(X)(\mathbb{T})$ are trivial. 
	\end{rmk}
	
	\section{Labelled $K$-algebras}\label{section: labelled k-algebras}

	Throughout this section, we let $K$ be a field equipped with a surjective valuation $\nu: K \rightarrow \mathbb{T}$.  $\mathcal{O}_K$ will denote the associated valuation ring. The following is the key definition in this section.
	
	\begin{mydef}$ $ \label{definition: labelled algebra}
		\begin{enumerate}
			\item 
			A \emph{labelled} $K$-algebra $(A,\phi)$ consists of a $K$-algebra $A$ together with an epimorphism $\phi: K[M] \to A$ for some monoid $M$ such that the map $M\to A$ is injective. For notational convenience, we will denote a labelled $K$-algebra $(A,\phi)$ simply by $A$. 
			\item 
			A \emph{monomial} of $A$ is an element of the form $\phi(km)$ for some $k\in K$ and $m\in M$.
			\item 
			A monomial is said to be a \emph{primitive monomial} if it has the form $\phi(m)$.	
			\item 
			We denote the set of monomials by $\mathbf{Mon}(A)$ and the set of primitive monomials by $\mathbf{PMon}(A)$.
			\item 
			A \emph{finitely generated monomial} $\mathcal{O}_K$-submodule of $A$ is an $\mathcal{O}_K$-submodule which is generated by finitely many monomials. 
		\end{enumerate}	
	\end{mydef}
	
	\begin{rmk}Labelled $\mathbb{N}$-algebra are called blueprints (c.f. \cite{oliver1}).  \cite{lorscheid2015scheme} links the study of tropical geometry to the monomial blueprint, which is essentially the labelled algebra we are working with.
	\end{rmk}
	
	Equipping the algebra $A$ with an epimorphism $K[M]\rightarrow A$ is necessary in order to tropicalize $\Spec A$. The additional condition that the map $M\rightarrow A$ is injective is probably unnecessary for most of what we do below, but allows us to skip some tedious steps in the proofs.  The below example should convince the reader that not too much is lost by adding this condition.
	
	\begin{myeg}
		Let $A=K[x,y] / (x^2 - y^3)$.  Let $M$ be the free monoid on 2 generators named $x$ and $y$.  We may equip $A$ with the canonical map $K[M]=K[x,y]\rightarrow A$.  Note that the map $M\rightarrow A$ is not injective.  Nonetheless, we may use this map to tropicalize $\Spec A$, and the coordinate semiring of the resulting tropical scheme is $\mathbb{T}[x, y]/\angles{x^2=y^3}$.  
		
		Now let $M'\subseteq A$ be generated by elements $x$ and $y$ with a relation $x^2=y^3$.  Then we have an epimorphism (in fact an isomorphism) $K[M']\rightarrow A$.  The corresponding map $M'\rightarrow A$ is injective, so we have a labelled $K$-algebra.  Tropicalizing $\Spec A$ with this structure gives $\mathbb{T}[M']\cong \mathbb{T}[x,y]/\angles{x^2=y^3}$.  So we were able to adjust the epimorphism $K[M]\rightarrow A$ to one giving the structure of a labelled $K$-algebra without changing the tropical scheme of interest.
	\end{myeg}
	
	For a given ring $A$, the set of ideals of $A$ forms an idempotent semiring. One may then expect that some properties of ideals of $A$ (more generally modules over $A$) to be captured purely in terms of certain structures of this associated idempotent semiring. We refer the interested readers to \cite{jun2020lattices} for detailed treatment of this line of thought along with an idempotent semiring analogue of Hochster's theorem on spectral spaces \cite{hochster1969prime}. To this end, we first prove the following.
	
	\begin{lem}
		Let $A$ be a labelled $K$-algebra.  The set of finitely generated monomial $\mathcal{O}_K$-submodules of $A$ is an idempotent semiring, where for two submodules $N, N'\subseteq A$, $N+N'$ is the smallest submodule containing both $N$ and $N'$, and $NN'$ is the smallest submodule containing $\{xy \mid x\in N, y\in N'\}$.
	\end{lem}
	\begin{proof}
		Showing that $\mathcal{O}_K$-submodules of $A$ form an idempotent semiring is  similar to showing that ideals form an idempotent semiring. Clearly $0$ and $\mathcal{O}_K$ are finitely generated monomial submodules. If $N,N'\subseteq A$ are finitely generated monomial submodules with generating sets $S,S'$, then $N+N'$ and $NN'$ are generated by $S\cup S'$ and $\{xy\mid x\in S, y\in S'\}$, both of which are finite sets of monomials.
	\end{proof}
	
	Recall that any $\mathbb{B}$-algebra $S$ comes with a canonical partial order; $a\leq b$ if and only if $a+b=b$ for $a,b \in S$. The following is another key definition in this section. 
	
	\begin{mydef}\label{definition: monomial valuation}Let $A$ be a labelled $K$-algebra.  Let $S$ be a $\mathbb{T}$-algebra.  A \emph{monomial valuation} on $A$ with values in $S$ is a monoid homomorphism $w: \mathbf{Mon}(A)\to S$ such that 
		\begin{enumerate}
			\item $w(k) = v(k)$ for all $k\in K$.
			\item If $x_1,\dots,x_n,y\in \mathbf{Mon}(A)$ and $y$ is an $\mathcal{O}_K$-linear combination of $x_1,\dots, x_n$, then 
			\begin{equation}\label{eq: monomial valuation}
				w(y) \leq w(x_1) + \cdots + w(x_n).
			\end{equation}
		\end{enumerate}
	\end{mydef}
	
	Valuations are essentially semiring homomorphisms whose source is the semiring of finitely generated submodules of a ring (c.f. \cite{giansiracusa2016equations}, \cite{macpherson2013skeleta}).  Analogously, we would like to link monomial valuations with homomorphisms from the semiring of finitely generated monomial $\mathcal{O}_K$-submodules.

	\begin{lem}\label{lemma: finitely generated fractional ideals is T}
		Let $A$ be a labelled $K$-algebra. The semiring of finitely generated fractional ideals of $K$ is isomorphic to $\mathbb{T}$. 
	\end{lem}
	\begin{proof}
		Let $\langle x \rangle$ be the fractional ideal of $K$ generated by $x \in K$. Then, we have
		\begin{equation}\label{eq: fractioanl principal}
			\langle x \rangle =\{y \in K \mid \nu(y) \leq \nu(x)\}
		\end{equation}
		since this condition is equivalent to $\frac{y}{x} \in \mathcal{O}_K$. This means the map from the set of principal fractional ideals to $\mathbb{T}$ given by
		\begin{equation}\label{eq: maps to T}
			\langle x \rangle \mapsto \nu(x)
		\end{equation}
		is well-defined, that is, it does not depend on the choice of generator. Furthermore, \eqref{eq: fractioanl principal} shows that $\nu(x) \leq \nu(y)$ implies that $\langle x \rangle$ is a subset of $\langle y \rangle$.
		
		Now, consider the fractional ideal generated by $x_1,\dots,x_n$. Without loss of generality, we may assume that $x_1$ has the largest valuation, and this implies that all the $x_i$ (and hence the entire fractional ideal) are contained in the fractional ideal generated by $x_1$. In particular, all finitely generated fractional ideals are principal. So, we have a monotonic bijection from finitely generated fractional ideals of $K$ to $\mathbb{T}$ given in \eqref{eq: maps to T}, and one can easily check that it is an isomorphism of semirings. 
	\end{proof}
	
	Let $A$ be a labelled $K$-algebra, and let $\mathbb{S}_{\mathrm{fgmon}}(A)$ be the semiring of finitely generated monomial $\mathcal{O}_K$-submodules of $A$. The semiring of finitely generated fractional ideals of $K$ is a subsemiring of $\mathbb{S}_{\mathrm{fgmon}}(A)$, and isomorphic to $\mathbb{T}$ by Lemma \ref{lemma: finitely generated fractional ideals is T}. By abuse of notation, we identify them, in particular, this means that $\mathbb{S}_{\mathrm{fgmon}}(A)$ is equipped with a natural $\mathbb{T}$-algebra structure. 
	
	\begin{pro}\label{proposition: monomial valuation submodules}
		With the same notation as above, the map $w_{univ}:\mathbf{Mon}(A)\rightarrow\mathbb{S}_{\mathrm{fgmon}}(A)$ which sends a monomial to the submodule it generates is a monomial valuation.  Furthermore, it is universal in the sense that every monomial valuation $w:\mathbf{Mon}(A)\rightarrow S$ factors as $w = f w_{univ}$ for a unique $\mathbb{T}$-algebra homomorphism $f:\mathbb{S}_{\mathrm{fgmon}}(A) \to S$.
	\end{pro}
	\begin{proof}	
		For $x_1, \ldots, x_n\in \mathbf{Mon}(A)$, we will write $\angles{x_1,\ldots, x_n}$ for the $\mathcal{O}_K$-submodule of $A$ generated by $\{x_1,\dots,x_n\}$. Now, for $k\in K$, one has that
		\[
		w_{univ}(k) = \angles{k} = v(k).
		\]
		Furthermore, $w_{univ}$ is clearly a monoid homomorphism. For the second axiom of Definition \ref{definition: monomial valuation}, let $x_1,\ldots,x_n,y\in \mathbf{Mon}(A)$ and $y$ be an $\mathcal{O}_K$-linear combination of $x_1,\dots,x_n$. In particular, 
		\[
		y \in \angles{x_1, \ldots, x_n} = \angles{x_1} + \ldots + \angles{x_n},
		\]
		and hence 
		\[
		w_{univ}(y) = \angles{y} \leq \angles{x_1} + \ldots + \angles{x_n} = w_{univ}(x_1) + \ldots + w_{univ}(x_n).
		\]
		This implies $w_{univ}$ is a monomial valuation.
		
		Now, to prove the universal property, let $w: \mathbf{Mon}(A)\to S$ be a monomial valuation.  Let $N$ be a finitely generated monomial $\mathcal{O}_K$-submodule of $A$, and $x_1,\ldots, x_n\in \mathbf{Mon}(A)$ be a set of generators of $N$. We first claim the following:
		\begin{equation}\label{eq: sup exists}
			\sup_{y\in N\cap \mathbf{Mon}(A)} w(y) = \sum_i w(x_i), 
		\end{equation}
		where the supremum is taken by using the canonical partial order of $\mathbb{T}$-algebras. In fact, it follows from Definition \ref{definition: monomial valuation} that any $y\in N\cap \mathbf{Mon}(A)$ satisfies $w(y) \leq \sum_i w(x_i)$. Furthermore, since $x_i \in N \cap \mathbf{Mon}(A)$, if $M$ is any other upper bound on the elements $w(y)$ for $y \in N \cap \mathbf{Mon}(A)$, then we have $w(x_i) \leq M$, and hence $\sum_i w(x_i) \leq M$ since $S$ is idempotent. This proves Equation~\eqref{eq: sup exists}, in particular, we can define the following function:
		\[
		f:\mathbb{S}_{\mathrm{fgmon}}(A)\rightarrow S, \quad f(N) = \displaystyle\sup_{y\in N\cap \mathbf{Mon}(A)} w(y).
		\]
		Note that if $N = \angles{x}$, the above equation implies $f(N) = w(x)$.  In particular, for $x\in\mathbf{Mon}(A)$, we have
		\[
		f(w_{univ}(x)) = f(\angles{x}) = w(x).
		\]
		Let $N'$ be a finitely generated monomial $\mathcal{O}_K$-submodule of $A$, generated by $y_1,\ldots, y_m\in \mathbf{Mon}(A)$. Then, $N+N'$ is generated by $x_1\ldots,x_n,y_1,\ldots,y_m$, and hence we have
		\[
		f(N+N') = \sum_iw(x_i)+\sum_jw(y_j) = f(N) + f(N').
		\]
		Similarly $NN'$ is generated by monomials of the form $x_i y_j$, and hence we have
		\[ 
		f(NN') = \sum_{i,j} w(x_i y_j) = \sum_{i,j}w(x_i)w(y_j)=(\sum_iw(x_i))(\sum_jw(y_j)) = f(N)f(N').
		\]
		By surjectivity of the valuation on $K$, any element of $\mathbb{T}$ has the form $v(k)$ for some $k\in K$, and recall we have identified $v(k)$ with $\angles{k}$.  So $f(v(k)) = f(\angles{k}) = w(k) = v(k)$.  Thus $f$ is a $\mathbb{T}$-algebra homomorphism.
		
		Finally, let $g$ be some other homomorphism such that $w = g w_{univ}$.  Then for any $x$, $f(\angles{x}) = f(w_{univ}(x)) = g(\angles{x})$.  Since every element of $\mathbb{S}_{\mathrm{fgmon}}(A)$ is a sum of elements of the form $\angles{x}$, this implies $f=g$.
	\end{proof}
	
	Next we link Definition \ref{definition: monomial valuation} to bend relations in \cite{giansiracusa2016equations}.
	
	\begin{lem}\label{lemma: monomial valuation bend relations}
		Axiom (2) of Definition \ref{definition: monomial valuation} is equivalent to the following condition (*) given all the other axioms in that definition:
		\medskip
		
		(*) Let $x_1, \ldots, x_n\in \mathbf{Mon}(A)$.  If $x_1 + \cdots + x_n = 0$, then for each $i=1,\ldots, n$, we have
		\begin{equation} 
			w(x_i) \leq \sum_{j\neq i} w(x_j). 
		\end{equation}
	\end{lem}
	\begin{proof}Let $w$ be a monomial valuation.  Then writing $x_i = \sum_{j\neq i} -  x_j$, the stated condition follows from Definition \ref{definition: monomial valuation}(2) since $\nu(1)=\nu(-1)$.
		
		Conversely suppose $w$ satisfies Definition \ref{definition: monomial valuation} except axiom (2), but satisfies the condition (*) from the statement of the lemma.  Let $x_1, \ldots, x_n, y\in \mathbf{Mon}(A)$.  Let $c_1, \cdots, c_n \in \mathcal{O}_K$ and suppose $y = c_1 x_1 + \cdots + c_n x_n$, or $(-y)+c_1x_1+\cdots+c_nx_n=0$. From the condition (*), we have
		\[ 
		w(-y) \leq \sum_i w(c_i x_i) = \sum_i w(c_i) w(x_i) \leq \sum_i w(x_i) 
		\]
		Since $w(-1) = 1$, we obtain axiom (2) of the definition.
	\end{proof}
	
	\begin{lem}\label{lemma: primitive monomial valuation}Let $A$ be a labelled $K$-algebra, and $S$ be a $\mathbb{T}$-algebra. Then there is a natural one-to-one correspondence between monomial valuations on A with values in S and monoid homomorphisms $f:\mathbf{PMon}(A)\rightarrow S$ satisfying the following condition (**):
		\medskip
		
		(**) Let $x_1, \ldots, x_n\in \mathbf{PMon}(A)$ and $k_1,\ldots, k_n\in K$.  If $k_1 x_1 + \cdots + k_n x_n = 0$ in $A$, then for each $i=1,\ldots, n$, we have
		\[ v(k_i) f(x_i) \leq \sum_{j\neq i} v(k_j) f(x_j) \]
		
	\end{lem}
	\begin{proof}
		Let $w$ be a monomial valuation.  The restriction of $w$ to $\mathbf{PMon}(A)$ is a monoid homomorphism.  For any $x\in \mathbf{PMon}(A)$ and $k\in K$, we have $w(kx) = w(k)w(x) = v(k)w(x)$.  Then condition (*) of Lemma \ref{lemma: monomial valuation bend relations} implies condition (**).  
		
		We have constructed a function from the set of monomial valuations to the set of monoid homomorphisms $\mathbf{PMon}(A)\rightarrow S$ satisfying (**).  It is injective because the equation $w(kx)=v(k)w(x)$ lets us reconstruct $w$ from its restriction to $\mathbf{PMon}(A)$.  
		
		Now given a map $f:\mathbf{PMon}(A)\rightarrow S$ which satisfies (**), let $w(kx)=v(k)f(x)$ for all $k\in K$ and $x\in \mathbf{PMon}(A)$. Then, $w$ is well-defined since if $kx=cy$, then $kx-cy=0$. Hence from the given condition $(**)$, we have that $v(k)f(x)=v(c)f(y)$. Furthermore, since $f(1)=1$, we have $w(k)=v(k)$.  It is routine to check $w$ is a monoid homomorphism.  Condition (*) of lemma \ref{lemma: monomial valuation bend relations} follows from (**).
	\end{proof}
	
	Let $(A,\phi)$ be a labelled $K$-algebra. An epimorphism $\phi:K[M] \to A$ provides an $\mathbb{F}_1$-structure to $A$, allowing us to perform scheme-theoretic tropicalization as in \cite{giansiracusa2016equations}.  
	
	\begin{pro}\label{proposition: monomial valuation tropicalization}
		Let $(A,\phi)$ be a labelled $K$-algebra.  Let $S$ be a $\mathbb{T}$-algebra.  Let $\mathrm{Trop}(A)$ be the coordinate semiring of the scheme-theoretic tropicalization of $\Spec A$ with respect to $\phi:K[M] \to A$. Then, there is a natural one-to-one correspondence:
		\[
		\{\textrm{homomorphisms }     \mathrm{Trop}(A) \rightarrow S \}	\longleftrightarrow \{\textrm{monomial valuations on $A$ with values in $S$} \}.
		\] 
	\end{pro}
	\begin{proof}
		Note that since the canonical map $M\rightarrow A$ is injective with image $\mathbf{PMon}(A)$, we have
		\[
		M\simeq \mathbf{PMon}(A) \textrm{ (as monoids).}
		\] 
		Since $\mathrm{Trop}(A)$ is the quotient of $\mathbb{T}[M]$ by bend relations, a homomorphism $\mathrm{Trop}(A) \rightarrow S$ corresponds to a monoid homomorphism $M \rightarrow S$ satisfying the relations in Lemma \ref{lemma: monomial valuation bend relations}, or equivalently to a monoid homomorphism $\mathbf{PMon}(A)\rightarrow S$ satisfying the relations in Lemma \ref{lemma: monomial valuation bend relations}. It follows from Lemma \ref{lemma: primitive monomial valuation} that these correspond to monomial valuations.
	\end{proof}
	\begin{cor}\label{corollary: tropical interpretation of monomial submodules}Let $(A,\phi)$ be a labelled $K$-algebra.  Let $\mathrm{Trop}(A)$ be the coordinate semiring of the scheme-theoretic tropicalization of $\Spec A$ with respect to $\phi:K[M] \to A$. Let $\mathbb{S}_{\mathrm{fgmon}}(A)$ be the semiring of finitely generated monomial $\mathcal{O}_K$-submodules of $A$. Then $\mathrm{Trop}(A) \simeq \mathbb{S}_{\mathrm{fgmon}}(A)$ as $\mathbb{T}$-algebras.
	\end{cor}
	\begin{proof}The two semirings have the same universal property according to Propositions \ref{proposition: monomial valuation submodules} and \ref{proposition: monomial valuation tropicalization}.
	\end{proof}
	
	\begin{myeg}\label{eg: canonical labelling}
		Let $A$ be a $K$-algebra.  By applying the universal property of the monoid algebra to the identity map, one obtains a homomorphism $K[A]\rightarrow A$.  This allows one to equip $A$ with the structure of a labelled $K$-algebra in a canonical way.  With this structure, all elements of $A$ are monomials.  Consequently finitely generated monomial $\mathcal{O}_K$-submodules are the same as finitely generated $\mathcal{O}_K$-submodules.  Furthermore, monomial valuations are the same as valuations on $A$ which extend the valuation on $K$.
	\end{myeg}
	
	\begin{myeg}
		One may tropicalize $\Spec A$ using the canonical labelled algebra structure in Example \ref{eg: canonical labelling}. Then one obtains $\Spec \mathbb{S}_\mathrm{fg}(A)$, where $\mathbb{S}_\mathrm{fg}(A)$ denotes the semiring of finitely generated $\mathcal{O}_K$-submodules of $A$. Note that the $\mathbb{T}$-valued points are just monomial valuations with values in $\mathbb{T}$.  Thus the $\mathbb{T}$-valued points of this tropicalization consists of all valuations on $A$ which extend the given valuation on $K$.  This tropical variety is essentially the Berkovich spectrum of $A$ except that we do not equip $A$ with a norm, and so do not require our multiplicative seminorms to be bounded.
	\end{myeg}

	\section{Lifting of line bundles} \label{section: Detropicalization of line bundles}
	
	As in the previous section, we let $K$ be a field equipped with a surjective valuation $\nu: K \rightarrow \mathbb{T}$ and $\mathcal{O}_K$ be the associated valuation ring throughout this section.
	
	In this section, we first prove for a $K$-algebra $A$, $\Spec A$ is homeomorphic to the saturated spectrum $\sSpec \mathbb{S}_\mathrm{fg}(A)$ (Definition \ref{definition: saturated spectrum}), where $\mathbb{S}_\mathrm{fg}(A)$ denotes the semiring of finitely generated $\mathcal{O}_K$-submodules of $A$. Then, we define a structure sheaf for the saturated spectrum $\Spec_s S$ of a semiring $S$ by using an inclusion $i:\Spec_s S \to \Spec S$. Equipped with this, we introduce a lifting of line bundles from $\Spec A$ to line bundles on $\sSpec \mathbb{S}_\mathrm{fg}(A)$. To be precise, when $A$ is a domain, we prove that one has the following isomorphism:
	\begin{equation}\label{eq: lifting}
		\Pic(\Spec A) \simeq \Pic(\sSpec \mathbb{S}_\mathrm{fg}(A)),
	\end{equation}
	When $A$ is a labelled algebra, \eqref{eq: lifting} implies that $\Pic(\Spec A)$ is isomorphic to $\Pic(\Spec_s \textrm{Trop}(A))$, where $\textrm{Trop}(A)$ is the coordinate semiring of the scheme-theoretic tropicalization of $\Spec A$. 
	
	\begin{rmk}
		We remark that the above results are different from the one that we obtained in \cite{jun2019picard}, where we prove that for a monoid scheme $X$ satisfying Condition \ref{condition: condition on open cover}, one has the following isomorphisms:
		\[
		\Pic(X) \simeq \Pic(X_k) \simeq \Pic(X_\mathbb{T}),
		\]
		where $k$ is a field and $\mathbb{T}$ is the tropical semifield. 
	\end{rmk}

	\begin{lem}\label{lemma: localization of submodules}
		Let $A$ be a $K$-algebra and let $\mathbb{S}_\mathrm{fg}(A)$ denote the semiring of finitely generated $\mathcal{O}_K$-submodules of $A$.  Let $f\in A$.  Then, we have
		\[
		\mathbb{S}_\mathrm{fg}(A)_{\angles{f}} \cong \mathbb{S}_\mathrm{fg}(A_f),
		\]	
		where $\langle f \rangle$ is the $\mathcal{O}_K$-submodule of $A$ generated by $f$.
	\end{lem}
	\begin{proof}
		Let $\phi: A\rightarrow A_f$ be the localization.  Define the image map
		\[
		\phi_*: \mathbb{S}_\mathrm{fg}(A)\rightarrow \mathbb{S}_\mathrm{fg}(A_f), \quad N \mapsto \phi(N).
		\]
		$\phi_*$ clearly preserves the zero submodule and $\angles{1}$.  Suppose $N$ and $N'$ are generated by $S$ and $S'$ respectively.  Then $N+N'$ is generated by $S\cup S'$, so $\phi_*(N+N')$ is generated by $\phi(S\cup S') = \phi(S)\cup \phi(S')$.  Since this is the same set generating $\phi_*(N)+\phi_*(N)$, $\phi_*$ preserves addition.  Similarly, we have $\phi(SS') = \phi(S)\phi(S')$, and hence $\phi_*$ preserves multiplication.
		
		Note that $\phi_*(\angles{f})$ is a unit, with inverse $\angles{\frac{1}{f}}\subseteq A_f$.  Thus, by the universal property of localization applied to $\mathbb{S}_\mathrm{fg}(A)_{\angles{f}}$, the map $\phi_*$ induces a (unique) homomorphism
		\[
		\psi:\mathbb{S}_\mathrm{fg}(A)_{\angles{f}} \rightarrow \mathbb{S}_\mathrm{fg}(A_f)
		\] 
		given by
		\[ 
		\psi\left(\frac{N}{\angles{f}^n}\right) = \angles{\frac{1}{f}}^n \phi_*(N). 
		\]
		Consider a principal $\mathcal{O}_K$-submodule $N\subseteq A_f$.  A generator of $N$ has the form $\frac{g}{f^n}$ for some $g \in A$.  Then one has
		\[ 
		N = \angles{\frac{1}{f}}^n\angles{\frac{g}{1}} = \angles{\frac{1}{f}}^n\phi_*(\angles{g}) = \psi\left(\frac{\angles{g}}{\angles{f}^n}\right). 
		\]
		
		Since every finitely generated submodule is a sum of principal submodules, which are all in the image of $\psi$, it follows that $\psi$ is surjective.  
		
		To check injectivity, we will use the claim that if $\phi_*(N)=\phi_*(N')$ then there is some $n$ such that $f^n N = f^n N'$.  If $x\in N$, then there is some $y\in N'$ such that $\phi(x)=\phi(y)$, so for some $k$ we have $f^k x = f^k y$.  The same holds with $N$ and $N'$ swapped.  Let $n$ be the largest value of $k$ needed to apply this remark to all the generators of $N$ and of $N'$ (which are finitely generated).  Then one sees $f^n N = f^n N'$, establishing the claim.
		
		If we are given two elements which map to the same place under $\psi$, we may find a common denominator and write the elements as $\frac{N}{\angles{f}^m}$ and $\frac{N'}{\angles{f}^m}$.  Then
		\[ \angles{\frac{1}{f}}^n \phi_*(N) = \angles{\frac{1}{f}}^n \phi_*(N') \]
		which implies $\phi_*(N) = \phi_*(N')$.  Then 
		\[\frac{N}{\angles{f}^m} = \frac{f^n N}{\angles{f}^{m+n}} = \frac{f^n N'}{\angles{f}^{m+n}} = \frac{N'}{\angles{f}^m} \]
		This shows $\psi$ is injective.
	\end{proof}
	
	\begin{mydef}\label{definition: saturated spectrum}
		Let $S$ be a semiring.  The \emph{saturated spectrum} $\sSpec S$ is the set of saturated prime ideals of $S$ together with the subspace topology induced by the inclusion $\sSpec S \subseteq \Spec S$.  If $j: \sSpec S \rightarrow \Spec S$ is the inclusion, the structure sheaf on $\sSpec S$ is defined to be $j^{-1}(\mathcal{O}_{\Spec S})$.  Here $j^{-1}$ denotes the pullback functor on sheaves.
	\end{mydef}
	
	More precisely, the closed sets in $\sSpec S$ are of the form $V(I)=\{\mathfrak{p} \in \sSpec S \mid I \subseteq \mathfrak{p}\}$. For more details on this topology, see \cite[\S 3]{jun2020lattices}. 
	
	\begin{lem}\label{lemma: pullback of pushforward}Let $X$ be a topological space, $Y\subseteq X$ be a subspace and $i: Y\rightarrow X$ be the inclusion.  Let $\mathcal{F}$ be a sheaf of either abelian semigroups or abelian groups on $Y$. Then there is an isomorphism $\mathcal{F}\cong i^{-1}(i_*(\mathcal{F}))$ which is natural in $\mathcal{F}$.
	\end{lem}
	\begin{proof}
		Applying the adjunction between $i^{-1}$ and $i_*$ gives a natural morphism $i^{-1}(i_*(\mathcal{F})) \rightarrow \mathcal{F}$ (compatible with the semigroup or group structure).  We must show this map induces bijections on stalks either explicitly or by noting the fact that the stalk functor on the category of sheaves of sets has no nontrivial endomorphisms implies that it suffices to show there is a natural bijection on stalks.
		
		Let $y\in Y$.  Since neighborhoods of $y$ inside $Y$ are precisely the sets $U\cap Y$ where $U$ ranges over neighborhoods of $i(y)$ inside $X$, we have the following, where limits are taken over neighborhoods of $i(y)$:
		\[ \mathcal{F}_y = \varinjlim \mathcal{F}(U\cap Y) = \varinjlim i_*\mathcal{F}(U) = (i_*\mathcal{F})_{i(y)} \]
		
		Since the inverse image functor preserves stalks 
		\[ i^{-1}(i_*(\mathcal{F}))_y = (i_*\mathcal{F})_{i(y)} = \mathcal{F}_y \]
	\end{proof}

	\begin{pro}$\sSpec$ is a contravariant functor from the category of semirings to the category of semiringed spaces.
	\end{pro}
	\begin{proof}
		For a semiring $S$, we let $j_S:\sSpec S \to \Spec S$ be the inclusion. Let $\phi: R\rightarrow S$ be a homomorphism of semirings.  Using functoriality of $\Spec$ and the fact that if $\mathfrak{p}$ is a saturated prime ideal then so is $\phi^{-1}(\mathfrak{p})$, we obtain the following continuous map:
		\[
		\sSpec\phi:\sSpec S\rightarrow \sSpec R.
		\]
		For the notational convenience, we let $Y=\Spec R$ and $X=\Spec S$. 
		Consider the following commutative diagram of topological spaces.
		\begin{equation}\label{eq: comm1}
			\begin{tikzcd}
				\sSpec S\arrow{d}[swap]{\sSpec \phi} \arrow{r}{j_S}
				& X \arrow{d}{\Spec \phi}\\
				\sSpec R \arrow{r}{j_R}
				&  Y
			\end{tikzcd}
		\end{equation}
		Since $\Spec \phi$ is a morphism of semiringed spaces, we have a natural morphism 
		\begin{equation}\label{eq: alpha}
			\alpha: (\Spec \phi)^{-1} \mathcal{O}_{Y} \rightarrow \mathcal{O}_{X}. 
		\end{equation}
		This induces the following morphism: 
		\begin{equation}\label{eq: def of structure map}
			j_S^{-1} (\alpha): j_S^{-1} (\Spec \phi)^{-1} \mathcal{O}_{Y} \rightarrow j_S^{-1} \mathcal{O}_{X} (= \mathcal{O}_{\sSpec S}).  
		\end{equation}
		From \eqref{eq: comm1}, we have 
		\[
		(\sSpec \phi)^{-1} \mathcal{O}_{\sSpec R} = (\sSpec \phi)^{-1} j_R^{-1} \mathcal{O}_{Y} = j_S^{-1} (\Spec \phi)^{-1} \mathcal{O}_{Y}.
		\]
		Thus we may regard $j_S^{-1} (\alpha)$ as a morphism 
		\[ 
		j_S^{-1}(\alpha): (\sSpec \phi)^{-1} \mathcal{O}_{\sSpec R}\rightarrow \mathcal{O}_{\sSpec S}. 
		\]
		Now, the adjunction between $(\sSpec \phi)^{-1}$ and $(\sSpec \phi)_*$ equips $\Spec_s \phi$ with the structure of a morphism of semiringed spaces:
		\[
		(\Spec_s \phi)^\#:\mathcal{O}_{\sSpec R}\rightarrow (\Spec_s \phi)_*\mathcal{O}_{\sSpec S}.
		\]
		
		Next, we check compatibility with composition. Let $f:R \to H$ and $g:H \to S$ be homomorphisms of semirings. It is clear that as a continuous map
		\[
		\sSpec (gf)=\sSpec(f)\sSpec(g).
		\]
		We note that the inverse image functor preserves stalks since the stalk can be considered as the inverse image sheaf along the inclusion of a point. The map \eqref{eq: alpha} induces $f_\mathfrak{p}:R_{f^{-1}(\mathfrak{p})} \to H_\mathfrak{p}$ on stalks even at unsaturated primes $\mathfrak{p}$. Thus $j_S^{-1}(\alpha)$ in \eqref{eq: def of structure map} induces $f_\mathfrak{p}$ on stalks at saturated primes, and the identifications made above do not change stalks. In particular, the composition is well-defined at each stalk and we have $g_\mathfrak{p}f_{g^{-1}(\mathfrak{p})}=(gf)_{\mathfrak{p}}$. This shows the compatibility with composition, and hence $\sSpec$ is a contravariant functor. 
	\end{proof}
	
	\begin{mythm}\label{theorem: topology of space of submodules}
		Let $A$ be a $K$-algebra.  Let $X=\Spec A$ and let $Y=\sSpec\mathbb{S}_\mathrm{fg}(A)$.  Then there is a canonical homeomorphism $X\cong Y$.  Under this homeomorphism, basic open sets $D(f)\subseteq \Spec A$ correspond to basic open sets $D(\angles{f})\subseteq \sSpec\mathbb{S}_\mathrm{fg}(A)$.  
	\end{mythm}
	\begin{proof}
		From the theory of algebraic lattices (see \cite[Theorem 3.28]{jun2020lattices}) we have a one-to-one correspondence 
		\[ \{\mathcal{O}_K\textrm{-submodules of }A\} \simeq \{\textrm{saturated subsemigroups of }\mathbb{S}_\mathrm{fg}(A)\}\]
		compatible with multiplication and with the lattice structure, which sends an $\mathcal{O}_K$-submodule $N$ to the saturated subsemigroup it generates in the semiring $\mathbb{S}_\mathrm{fg}(A)$.  
		
		Note that the ideals of $A$ are precisely the $\mathcal{O}_K$-submodules $I\subseteq A$ such that for any other submodule $N\subseteq A$ one has $IN\subseteq I$. Under the above correspondence these correspond to saturated subsemigroups $J\subseteq \mathbb{S}_\mathrm{fg}(A)$ such that for any other saturated subsemigroup $P\subseteq \mathbb{S}_\mathrm{fg}(A)$, one has $JP\subseteq J$.  But such saturated subsemigroups are precisely the saturated ideals.  Thus one has a one-to-one correspondence
		\[ \{\textrm{ideals of }A\} \simeq \{\textrm{saturated ideals of }\mathbb{S}_\mathrm{fg}(A)\}\]
		which preserves multiplication and the lattice structure. For future use, we note that if an $\mathcal{O}_K$-submodule $N\subseteq A$ corresponds to a saturated subsemigroup $S\subseteq \mathbb{S}_\mathrm{fg}(A)$, then the ideal generated by $N$ in $A$ corresponds to the ideal generated by $S$ in $\mathbb{S}_\mathrm{fg}(A)$.  In particular, the ideal generated by $f\in A$ is generated by the submodule $\angles{f}\subseteq A$, so corresponds to the saturated ideal generated by $\angles{f}\in \mathbb{S}_\mathrm{fg}(A)$.
		
		The lattice of ideals together with the multiplicative structure on ideals determines the Zariski topology (see \cite{takagi2010construction}).  So we have a homeomorphism 
		\[
		\Spec A\simeq \sSpec\mathbb{S}_\mathrm{fg}(A).
		\]
		Furthermore, since the ideal of $A$ generated by $f$ corresponds to the saturated ideal of $\mathbb{S}_\mathrm{fg}(A)$ generated by $\angles{f}$, their associated closed sets correspond under this homeomorphism.  Taking complements, we note that the homeomorphism identifies the open sets $D(f)$ and $D(\angles{f})$.
	\end{proof}

	\begin{lem}\label{lemma: structure sheaf of Y}
		Let $A$ be a $K$-algebra which is a domain, and $Y=\sSpec\mathbb{S}_\mathrm{fg}(A)$. For each basic open subset $D(\angles{f})$, we define
		\[
		\mathcal{F}(D(\angles{f})) =\mathbb{S}_\mathrm{fg}(A_f).
		\]
		Then $\mathcal{F}$ is a sheaf on $Y$.
	\end{lem}
	\begin{proof}
		We only have to check the sheaf axiom for $\mathcal{U}=\{D(\angles{f})\}_{f \in A}$ since then $\mathcal{F}$ will uniquely extend to define a sheaf on $Y$. We first check it is well-defined and is a presheaf.  If $D(\angles{f})\subseteq D(\angles{g})$, then by Theorem \ref{theorem: topology of space of submodules}, we have $D(f)\subseteq D(g)$ so we get a map $A_g \rightarrow A_f$.  This induces a map $\mathbb{S}_\mathrm{fg}(A_g)\rightarrow \mathbb{S}_\mathrm{fg}(A_f)$, which we choose to be our restriction map.  Clearly the composition of restriction maps is a restriction map, so the presheaf axiom is satisfied.  If $D(\angles{f})= D(\angles{g})$, then we even have $A_f = A_g$ (as subsets of $\Frac(A)$), so our restriction map is the identity.  Thus $\mathcal{F}(D(\angles{f}))=\mathcal{F}(D(\angles{g}))$, i.e. $\mathcal{F}(D(\angles{f}))$ is well-defined.
		
		To check the sheaf axiom, we first consider the case when $f=1$, that is, the open subsets $D(\angles{f_i})$ form an open cover of $Y$. Suppose that we have a section $h_i$ over each $D(\angles{f_i}) = \mathbb{S}_\mathrm{fg}(A_{f_i})$, and that we have $h,h' \in \mathbb{S}_\mathrm{fg}(A)$ such that $h_{D(\angles{f_i})}=h_{D(\angles{f_i})}'=h_i$. In other words, $h,h'$ are finitely generated $\mathcal{O}_K$-submodules of $A$ such that $h_{f_i}=h'_{f_i}$, where $h_{f_i}$ is the image of $h$ in $A_{f_i}$. But, since $A$ is a domain, the localization map $A \to A_{f_i}$ is injective, and hence $h=h'$. In particular, if the sections $h_i$ can be glued, then it should be unique.

		Next, we prove that the sections $h_i$ can be glued. In fact, we can find a common denominator for the generators of these submodules $h_i$ in such a way that the sections $h_i$ can be written as $N_i/f_i^{k_i}$, where $N_i$ is a finitely generated $\mathcal{O}_K$-submodule of $A$.  For any $i$ and $j$, choose $g_{ij}$ such that 
		\[
		D(\angles{g_{ij}}) \subseteq D(\angles{f_i}) \cap D(\angles{f_j}).
		\]
		Then $N_i / f_i^{k_i}$ and $N_j / f_j^{k_j}$ are equal as submodules of $A_{g_{ij}}$ and thus as submodules of $\Frac(A)$; call this common submodule $M$.  Now choose $a_1,\dots,a_n$ such that 
		\[
		a_1f_1^{k_1} +\cdots + a_n f_n^{k_n} = 1
		\]
		inside $A$; this is possible since the $D(f_i)$ are an open cover of $\Spec A$ by Theorem \ref{theorem: topology of space of submodules}. Now $M$ is a subset of (but, in general, not equal to) the following module:
		\begin{equation}\label{eq: module}
			a_1f_1^{k_1}M + \cdots + a_nf_n^{k_n}M = a_1 N_1 +\cdots + a_n N_n.
		\end{equation}
		But the module \eqref{eq: module} is a finitely generated $\mathcal{O}_K$-submodule of $A$, and hence $M$ is a finitely generated $\mathcal{O}_K$-submodule of $A$, so the sections can be glued. 
		
		Finally, the general case of $f$, we can set $A'=A_f$, $f'_i=ff_i$ reducing the general case to the case we have just proved. 
	\end{proof}

	\begin{mythm}\label{theorem: sheaf of space of submodules}
		Let $A$ be a $K$-algebra which is a domain. Let $X=\Spec A$ and let $Y=\sSpec\mathbb{S}_\mathrm{fg}(A)$ be the saturated spectrum of the semiring $\mathbb{S}_\mathrm{fg}(A)$.  %Recall there is a canonical homeomorphism $X\cong Y$ identifying basic open sets $D(f)\subseteq \Spec A$ with basic open sets $D(\angles{f})\subseteq \sSpec\mathbb{S}_\mathrm{fg}(A)$.  
		With the homeomorphism $X\simeq Y$ in Theorem \ref{theorem: topology of space of submodules}, one has isomorphisms $\mathcal{O}_{Y}(D(\angles{f})) \cong \mathbb{S}_\mathrm{fg}\left(\mathcal{O}_X(D(f))\right)$.  These isomorphisms are compatible with restriction in the sense that the following diagram commutes for any inclusion of basic opens $D(f)\subseteq D(g)\subseteq X$:
		\begin{equation}\label{eq: restriction compatibility of spectrum of submodules}
			\begin{tikzcd}
				\mathcal{O}_{Y}(D(\angles{g})) \arrow{d}{\simeq} \arrow{r}{}
				& \mathcal{O}_{Y}(D(\angles{f})) \arrow{d}{\simeq} \\
				\mathbb{S}_\mathrm{fg}\left(\mathcal{O}_X(D(g))\right)\arrow{r}{}
				& \mathbb{S}_\mathrm{fg}\left(\mathcal{O}_X(D(f))\right)
			\end{tikzcd}
		\end{equation}
		where the top arrow is the restriction, and the bottom arrow sends a submodule to its image under the restriction map $\mathcal{O}_X(D(g))\rightarrow\mathcal{O}_X(D(f))$.
	\end{mythm}
	\begin{proof}Let $Z = \Spec\mathbb{S}_\mathrm{fg}(A)$ and let $j: Y\rightarrow Z$ be the inclusion.  To avoid confusion, we will use $D_Y$ and $D_Z$ rather than $D$ in our notation for basic open sets of $Y$ or $Z$. Let $\mathcal{F}$ be the sheaf defined in Lemma \ref{lemma: structure sheaf of Y}. We claim that there is a canonical isomorphism $j_*(\mathcal{F}) \cong \mathcal{O}_Z$. In fact, on basic open sets, one has
		\[ 
		j_*(\mathcal{F})(D_Z(\angles{f})) = \mathcal{F}(D_Y(\angles{f}) = \mathbb{S}_{\textrm{fg}}(A_f) \cong (\mathbb{S}_{\textrm{fg}}(A))_{\angles{f}} = \mathcal{O}_Z(D_Z(\angles{f}).
		\]
		
		We need to verify compatibility of these isomorphisms with restriction maps, i.e., commutativity of the following diagram whenever $D_Z(\angles{g}) \subseteq D_Z(\angles{f})$.  Note that in this case, we also have $D_Y(\angles{g}) \subseteq D_Y(\angles{f})$.
		
		\begin{equation}
			\begin{tikzcd}
				j_*(\mathcal{F})(D_Z(\angles{f})) \arrow{d}{} \arrow{r}{\simeq}
				& \mathcal{F}(D_Y(\angles{f}) \arrow{d}{} \arrow{r}{\simeq}
				& \mathcal{O}_Z(D_Z(\angles{f}) \arrow{d}{}\\
				j_*(\mathcal{F})(D_Z(\angles{g})) \arrow{r}{\simeq}
				& \mathcal{F}(D_Y(\angles{g}) \arrow{r}{\simeq}
				& \mathcal{O}_Z(D_Z(\angles{g}) 
			\end{tikzcd}
		\end{equation}
		
		The left square commutes by the construction of the pushforward sheaf.  The right square can be expressed as follows:
		\begin{equation}
			\begin{tikzcd}
				\mathbb{S}_{\textrm{fg}}(A_f) \arrow{d}{} \arrow{r}{\simeq}
				& \mathbb{S}_{\textrm{fg}}(A)_{\angles{f}} \arrow{d}{}\\
				\mathbb{S}_{\textrm{fg}}(A_g) \arrow{r}{\simeq}
				& \mathbb{S}_{\textrm{fg}}(A)_{\angles{g}}
			\end{tikzcd}
		\end{equation}
		
		By the description of the restriction map of $\mathcal{F}$ given in the proof of Lemma \ref{lemma: structure sheaf of Y}, the left arrow is induced by the localization map $A_f\rightarrow A_g$.  By definition of $\mathcal{O}_Z$, the right arrow is the localization map $\mathbb{S}_{\textrm{fg}}(A)_{\angles{f}}\rightarrow \mathbb{S}_{\textrm{fg}}(A)_{\angles{g}}$.  The top and bottom arrows are described in Lemma \ref{lemma: localization of submodules}.
		
		To check this square commutes, we let $M\in \mathbb{S}_{\textrm{fg}}(A_f)$ be a principal $\mathcal{O}_K$-submodule and we check that the image of $M$ inside $\mathbb{S}_{\textrm{fg}}(A)_{\angles{g}}$ is the same regardless of which path along the square we take.  This is enough because $\mathbb{S}_{\textrm{fg}}(A_f)$ is generated by principal $\mathcal{O}_K$-submodules.  We may write $M$ as $M=\angles{\frac{h}{f^k}}$ for some $h\in A$.  Also note that since $D(g)\subseteq D(f)$, we have $g^r=af$ for some $a\in A$ and $r \in \mathbb{N}$. Since $D(g)=D(g^r)$, we may assume that $g=af$.
		
		Following the diagram clockwise, $\angles{\frac{h}{f^k}}$ maps to $\frac{\angles{h}}{\angles{f}^k} \in \mathbb{S}_{\textrm{fg}}(A)_{\angles{f}}$, which maps to $\frac{\angles{a}^k \angles{h}}{\angles{g}^k} \in \mathbb{S}_{\textrm{fg}}(A)_{\angles{g}}$ thanks to the identity $\angles{g}=\angles{a}\angles{f}$.  Following the diagram counterclockwise, $\angles{\frac{h}{f^k}}$ maps to $\angles{\frac{a^k h}{g^k}}\in \mathbb{S}_{\textrm{fg}}(A_g)$ which maps to $\frac{\angles{a^k h}}{\angles{g}^k} \in \mathbb{S}_{\textrm{fg}}(A)_{\angles{g}}$.  By inspection, this is the same thing we got when following the diagram clockwise.  This establishes commutativity of the above diagram and completes the proof of the claim. In particular, we have 
		\[ \mathcal{O}_Y = j^{-1}(\mathcal{O}_Z) \cong j^{-1}(j_*(\mathcal{F})) \cong \mathcal{F} \]
		using the definition of $\mathcal{O}_Y$, the claim we just established, and Lemma \ref{lemma: pullback of pushforward}.
		
		We now turn to the proof of the theorem, noting that throughout the statement of the theorem we may replace $\mathcal{O}_Y$ with the isomorphic sheaf $\mathcal{F}$.  By the definitions of $\mathcal{F}$ and $\mathcal{O}_X$, we have $\mathcal{F}(D(\angles{f})) = \mathbb{S}_{\textrm{fg}}(A_f) = \mathbb{S}_{\textrm{fg}}(\mathcal{O}_X(D(f)))$.  Similarly, using the fact that restriction maps in $\mathcal{O}_X$ are just localizations and the description of the restriction maps in $\mathcal{F}$ from the proof of Lemma \ref{lemma: structure sheaf of Y}, we immediately get the commutativity of \eqref{eq: restriction compatibility of spectrum of submodules}.
	\end{proof}

	\begin{mythm}\label{theorem: invertible submodules}
		Let $A$ be a $K$-algebra which is a domain, and some point $\mathfrak{p}$ of $\Spec A$ has the residue field isomorphic to $K$ (as a $K$-algebra). Let $N\subseteq A$ be a finitely generated $\mathcal{O}_K$-submodule.  Suppose there is some $\mathcal{O}_K$-submodule $N'\subseteq A$ such that $NN'=\mathcal{O}_K$.  Then we have the following:
		\begin{enumerate}
			\item 
			$N$ is principal and generated by a unit of $A$.
			\item 
			If in addition to the above, $A$ is a labelled $K$-algebra and $N$ is a finitely generated monomial $\mathcal{O}_K$-submodule, then $N$ is generated by an element of $\mathbf{Mon}(A)^\times$.	
		\end{enumerate}	
	\end{mythm}
	\begin{proof}
		$(1)$: First, we claim that every nonzero element of $N$ is a unit of $A$.  Note $N'\neq 0$, so pick some nonzero $g\in N'$.  If $f\in N$ is nonzero, $fg\neq 0$ and $fg\in \mathcal{O}_K\subseteq K$.  So $fg$ is a unit and hence $f$ is a unit.
		
		For the remainder of the proof fix some nonzero $f\in N$ (this is possible since $N\neq 0$). For the proof of $(2)$, if $N$ is a finitely generated monomial $\mathcal{O}_K$-submodule of a labelled $K$ algebra $A$, then we will specifically choose $f$ to be a monomial, otherwise $f$ is arbitrary.
		
		Let $g\in N$.  Since $f$ is a unit, $f(\mathfrak{p})\neq 0$.  Let $c=\frac{g(\mathfrak{p})}{f(\mathfrak{p})}\in K$.  Write $c=\frac{a}{b}$ with $a,b\in \mathcal{O}_K$.  Then 
		\[
		bg(\mathfrak{p}) - af(\mathfrak{p}) = 0,
		\]
		and hence $bg - af$ is a non-unit.  But $bg - af\in N$, so $bg - af = 0$.  Then $g = cf$.
		
		Since $N$ is finitely generated and each element has the form $cf$ for $c\in K$, we may pick $c_1,\ldots, c_n$ such that $N$ is generated by $c_1f,\ldots, c_nf$.  Without loss of generality, we may assume that they are ordered in such a way that $c_1$ has maximal valuation, which implies that for each $i$, $\frac{c_i}{c_1}\in \mathcal{O}_K$.  This implies $c_i f$ is in the submodule generated by $c_1f$, so $N$ is generated by $c_1 f$.  Since $f$ is a unit and $c_1$ is in $K$ (and $N\neq 0$), $N$ is generated by a unit.  
		
		$(2)$: If we chose $f$ to be a monomial above, then $N$ is generated by a unit which is a monomial.
	\end{proof}
	
	\begin{lem}\label{lemma: naturality of principal submodules}
		Let $A$ be a $K$-algebra. Consider the map $\phi_A:A\rightarrow \mathbb{S}_\mathrm{fg}(A)$, where $\phi_A(a)=\angles{a}$, the $\mathcal{O}_K$-submodule generated by $a$. Then, $\phi_A$ is natural in $A$.  Consequently so is the induced map on units $A^\times\rightarrow (\mathbb{S}_\mathrm{fg}(A))^\times$.
	\end{lem}
	\begin{proof}Let $f:A\rightarrow B$ be a homomorphism of $K$-algebras.  Given $a\in A$, one has 
		\[ 
		\phi_B(f(a)) = \angles{f(a)} = f_*(\angles{a}) = f_*(\phi_A(a)),
		\]
		where $f_*$ is the map induced by $f$ under the functor $\mathbb{S}_\mathrm{fg}$.
	\end{proof}
	
	\begin{mythm}\label{theorem: picard isomorphism}
		Let $A$ be a finitely generated $K$-algebra which is a domain. Then one has an isomorphism
		\[
		\Pic(\Spec A) \simeq \Pic(\sSpec \mathbb{S}_\mathrm{fg}(A)). 
		\] 
	\end{mythm}
	\begin{proof}
		If $f\in A$ is nonzero, then $A_f$ is also a domain and is a reduced finitely generated $K$-algebra.  By the weak Nullstellensatz, the residue field at any closed point (in particular at some closed point) of $A$ is $K$.  Thus $A_f$ satisfies the conditions of Theorem \ref{theorem: invertible submodules}.
		
		We will let $X=\Spec A$ and use Theorem \ref{theorem: topology of space of submodules} to identify $Y=\sSpec \mathbb{S}_\mathrm{fg}(A)$ with $X$ as topological spaces.  
		
		Our first step is to construct an epimorphism of sheaves $\eta: \mathcal{O}_X^\times \rightarrow \mathcal{O}_Y^\times$.  We start by defining $\eta$ on basic open sets.  For $U=D(f) \subseteq X$, we have a homomorphism 
		\[
		\phi_f : \mathcal{O}_X^\times(U) \rightarrow \mathbb{S}_\mathrm{fg}(\mathcal{O}_X(U))^\times
		\]
		mapping $a\in \mathcal{O}_X^\times(U)$ to the submodule $\phi_f(a)=\angles{a}$ it generates.  Theorem \ref{theorem: invertible submodules} tells us that $\phi_f$ is an epimorphism. By Theorem \ref{theorem: sheaf of space of submodules}, we have an isomorphism
		\[
		\tau_f: \mathbb{S}_\mathrm{fg}(\mathcal{O}_X(U))^\times\rightarrow \mathcal{O}_Y^\times(D(\angles{f})).
		\]
		We define $\eta_{D(f)} = \tau_f \circ \phi_f$.  Lemma \ref{lemma: naturality of principal submodules} tells us the maps $\phi_f$ are compatible with restriction from $D(f)$ to a smaller open set $D(g)$, while Theorem \ref{theorem: sheaf of space of submodules} tells us the same thing for the maps $\tau_f$.  Thus we may conclude the same for $\eta$.  By \cite[Lemma 6.30.10]{stacks-project}, our definition of $\eta$ extends uniquely to a morphism of sheaves.  Furthermore, the fact that $\eta$ induces an epimorphism of sections over any basic open set implies it is a sheaf-theoretic epimorphism.
		
		Next we compute the kernel of $\eta$.  Note that $\ker \eta_{D(f)} = \ker \phi_f$.  Let $x \in \ker \phi_f \subseteq \mathcal{O}_X^\times(D(f)) = A_f^\times$.  Then $\angles{x} = \angles{1}$.  Note that since $A$ is a $K$-algebra, the embedding $\mathcal{O}_K\rightarrow A$ is injective and we can identify $\angles{1}$ with $\mathcal{O}_K$.  The above can be written $x\mathcal{O}_K = \mathcal{O}_K$, which implies $x\in \mathcal{O}_K^\times$.  Conversely, if $x\in \mathcal{O}_K^\times$, then $\angles{x}=\angles{1}$ so $x\in\ker\phi_f$.  Thus $\ker\eta_{D(f)} = \mathcal{O}_K^\times$.  This defines an isomorphism of sheaves on our basis of open sets, so by \cite[Lemma 6.30.10]{stacks-project}, $\ker\eta$ is the constant sheaf $\underline{\mathcal{O}_K^\times}$.
		
		We thus have an exact sequence
		\[ 0 \rightarrow \underline{\mathcal{O}_K^\times} \rightarrow \mathcal{O}_X^\times \rightarrow \mathcal{O}_Y^\times \rightarrow 0 \]
		Since $X$ is irreducible, constant sheaves such as $\underline{\mathcal{O}_K^\times}$ are acyclic.  The long exact sequence and the homeomorphism $X\cong Y$ give us the following isomorphism: 
		\[
		H^1(X, \mathcal{O}_X^\times)\cong H^1(X, \mathcal{O}_Y^\times) = H^1(Y, \mathcal{O}_Y^\times),
		\]
		and this is the desired isomorphism of Picard groups.
	\end{proof}
	
	Now, from Corollary \ref{corollary: tropical interpretation of monomial submodules} and Example \ref{eg: canonical labelling}, we obtain the following.

	\begin{cor}
		Let $A$ be a $K$-algebra, viewed as a labelled $K$-algebra with a canonical labelled algebra structure in Example \ref{eg: canonical labelling}. Then, one has 
		\[
		\Pic(\Spec A) \simeq \Pic(\sSpec \mathrm{Trop}(A)),
		\]	
		where $ \mathrm{Trop}(A)$ is the coordinate semiring of the scheme-theoretic tropicalization of $\Spec A$ with respect to the canonical labelled algebra structure. 
	\end{cor}

	For a semiringed space $(X,\mathcal{O}_X)$, and $\mathcal{O}_X$-modules $\mathcal{F}$ and $\mathcal{G}$, one can define the tensor product $\mathcal{F}\otimes_{\mathcal{O}_X}\mathcal{G}$ as in the classical case. One can further use the same arguments as in the ring case to see that many properties hold in this case as well (for instance, see \cite[\S 3]{jun2017vcech}).
	
	\begin{mydef}\label{definition: $O_X$-module}
		Let $(X, \mathcal{O}_X)$ and $(Y, \mathcal{O}_Y)$ be semiringed spaces.  Given a morphism of semiringed spaces $\phi: X \rightarrow Y$ and a sheaf $\mathcal{F}$ of $\mathcal{O}_Y$-modules, let $\phi^*\mathcal{F}$ be the sheaf of $\mathcal{O}_X$-modules given by 
		\[ 
		\phi^*\mathcal{F} = \phi^{-1}\mathcal{F} \otimes_{\phi^{-1}\mathcal{O}_Y} \mathcal{O}_X 
		\]
	\end{mydef}
	
	One may apply a similar argument as in the ring case to see that Definition \ref{definition: $O_X$-module} is functorial.

	\begin{lem}
		Let $f: X\rightarrow Y$ be a morphism of semiringed spaces and $\mathcal{L}$ be a line bundle on $Y$.  Then $f^* \mathcal{L}$ is a line bundle on $X$.
	\end{lem}
	\begin{proof}We first consider the case of the trivial bundle $\mathcal{L} = \mathcal{O}_Y$.  We have 
		\[ f^* \mathcal{O}_Y = (f^{-1} \mathcal{O}_Y) \otimes_{(f^{-1} \mathcal{O}_Y)} \mathcal{O}_X = \mathcal{O}_X \]
		which gives the result in the case of the trivial bundle.
		
		To go further, we need to understand pullbacks along open sets.  If $Z$ is any topological space with an open subset $U\subseteq Z$ and if $i: U\rightarrow Z$ is the inclusion, then for every sheaf $\mathcal{F}$ of abelian semigroups on $Z$ we have $i^{-1} \mathcal{F} = \mathcal{F}\big|_U$.  Furthermore if $Z$ is a semiringed space, we may equip $U$ with the structure sheaf $\mathcal{O}_U = \mathcal{O}_Z \big|_U$.  Then for a sheaf $\mathcal{F}$ of $\mathcal{O}_Z$-modules we have
		\[ i^* \mathcal{F} = i^{-1} \mathcal{F} \otimes_{i^{-1} \mathcal{O}_Z} \mathcal{O}_U = i^{-1} \mathcal{F} = \mathcal{F}\big|_U. \]
		
		Now choose $U\subseteq Y$ such that $\mathcal{L}\big|_U \cong \mathcal{O}_Y\big|_U$.  Let $i: U\rightarrow Y$ and $j: f^{-1}(U) \rightarrow X$ be inclusion maps and let $g: f^{-1}(U)\rightarrow U$ be the map induced by $f$.  We regard $U$ and $f^{-1}(U)$ as semiringed spaces by restriction, as above. By construction, $ig=fj$. This gives
		\[ 
		(f^* \mathcal{L})\big|_{f^{-1}(U)} = j^*f^* \mathcal{L} = g^* i^* \mathcal{L} \cong g^* \mathcal{O}_U. \]
		
		In the special case $\mathcal{L} = \mathcal{O}_Y$, we have seen $f^* \mathcal{O}_Y \cong \mathcal{O}_X$.  Hence 
		\[
		\mathcal{O}_{f^{-1}(U)} = \mathcal{O}_X\big|_{f^{-1}(U)} \cong (f^* \mathcal{O}_Y)\big|_{f^{-1}(U)} \cong g^* \mathcal{O}_U. 
		\]
		Combining this with the more general case above, we have that if $\mathcal{L}$ is trivial on $U$, then 
		\[
		(f^* \mathcal{L})\big|_{f^{-1}(U)} \cong g^* \mathcal{O}_U \cong \mathcal{O}_{f^{-1}(U)}. 
		\]
		In other words, $f^* \mathcal{L}$ is trivial on $f^{-1}(U)$. Now if we pick a trivializing open cover $\{ U_i \}$ for a line bundle $\mathcal{L}$ then $\{f^{-1}(U_i)\}$ is a trivializing open cover for $f^* \mathcal{L}$.
	\end{proof}
	
	We are now ready to describe how to construct line bundles on a classical variety from tropical line bundles.  Recall from Corollary \ref{corollary: tropical interpretation of monomial submodules} that the tropicalization of $\Spec A$ is $\Spec \mathbb{S}_\mathrm{fgmon}(A)$ for a labelled $K$-algebra $A$.
	
	\begin{mydef}\label{definition: lifting def}
		Let $A$ be a labelled $K$-algebra which is finitely generated as an algebra and has no zero-divisors. 
		\begin{enumerate}
			\item 
			The \emph{saturated lifting map} is the map
			\[
			\tau_s: \Pic(\sSpec \mathbb{S}_\mathrm{fgmon}(A))\rightarrow \Pic(\Spec A)
			\]
			given by composing the isomorphism $\Pic(\sSpec \mathbb{S}_\mathrm{fg}(A))\cong \Pic(\Spec A)$ of Theorem \ref{theorem: picard isomorphism} with the pullback map
			\[
			\Pic(\sSpec \mathbb{S}_\mathrm{fgmon}(A))\rightarrow \Pic(\sSpec \mathbb{S}_\mathrm{fg}(A))
			\]
			induced by the inclusion $\mathbb{S}_\mathrm{fgmon}(A)\rightarrow \mathbb{S}_\mathrm{fg}(A)$.
			\item 
			The \emph{unsaturated lifting map} 
			\[
			\tau: \Pic(\Spec \mathbb{S}_\mathrm{fgmon}(A))\rightarrow \Pic(\Spec A)
			\]
			is given by composing $\tau_s$ with the pullback map
			\[
			\Pic(\Spec \mathbb{S}_\mathrm{fgmon}(A))\rightarrow \Pic(\sSpec \mathbb{S}_\mathrm{fgmon}(A))
			\]
			induced by the inclusion $\sSpec \mathbb{S}_\mathrm{fgmon}(A)\subseteq\Spec \mathbb{S}_\mathrm{fgmon}(A)$.
		\end{enumerate}
	\end{mydef}

	\begin{myeg}
		In the setting of Example \ref{eg: canonical labelling}, $\mathbb{S}_{\textrm{fgmon}}(A) = \mathbb{S}_{\textrm{fg}}(A)$. One may easily use this fact to check that in this situation, the saturated lifting map is an isomorphism.  
	\end{myeg}

	\begin{pro}\label{proposition: lifting pro}
		Let $A$ be a labelled $K$-algebra which is finitely generated as an algebra and has no zero-divisors.  Let $Y = \sSpec \mathbb{S}_\mathrm{fgmon}(A)$ and $Z = \Spec \mathbb{S}_\mathrm{fgmon}(A)$.  Suppose that for every line bundle $\mathcal{L}$ on $Y$ there is an open cover $\{U_i\}$ of $Z$ such that $\mathcal{L}$ is trivial on each $U_i \cap Y$.  Then the saturated and unsaturated lifting maps have the same image.
	\end{pro}
	\begin{proof}Let $j: \sSpec \mathbb{S}_\mathrm{fgmon}(A)\rightarrow\Spec \mathbb{S}_\mathrm{fgmon}(A)$ be the inclusion map.  Then $\tau = \tau_s \circ j^*$.  It follows that it suffices to show $j^*$ is surjective. Let $\mathcal{L}$ be a line bundle on $Y$.
		
		We claim that $j_* \mathcal{L}$ is a line bundle on $Z$.  In the proof of Theorem \ref{theorem: sheaf of space of submodules}, we had isomorphisms $j_*(\mathcal{F})\cong \mathcal{O}_Z$ and $\mathcal{F}\cong \mathcal{O}_Y$.  Together, these give an isomorphism $j_*(\mathcal{O}_Y) \cong \mathcal{O}_Z$, which covers the case of the trivial bundle.  
		
		Let $U\subseteq Z$ be an open set such that $\mathcal{L}$ is trivial on $j^{-1}(U)$.  Let $V\subseteq U$ be open.  We have
		\[ 
		(j_* \mathcal{L})(V) = \mathcal{L}(j^{-1}(V)) \cong \mathcal{O}_Y(j^{-1}(V)) = (j_* \mathcal{O}_Y)(V) = \mathcal{O}_Z(V). 
		\]
		
		It is routine to check that this sequence of isomorphisms is compatible with restriction maps, so $(j^* \mathcal{L})\big|_U = \mathcal{O}_Z\big|_U$.  In particular, $j_* \mathcal{L}$ is trivial on $U$.

		Choose an open cover $\{U_i\mid i\in\Gamma\}$ for some indexing set $\Gamma$ of $Z$ such that $\mathcal{L}$ is trivial on each set $j^{-1}(U_i) = U_i \cap Z$.  The above shows $j_*(\mathcal{L})$ is trivial on each $U_i$, establishing the claim.  Now we have
		\[ 
		j^*(j_*(\mathcal{L})) = j^{-1}(j_*(\mathcal{L})) \otimes_{j^{-1}\mathcal{O}_Z} \mathcal{O}_Y \cong j^{-1}(j_*(\mathcal{L})) \cong \mathcal{L}. 
		\]
		
		This shows $j^*$ is a surjection on Picard groups.
	\end{proof}

	\bibliography{Vectorbundle}\bibliographystyle{alpha}

\end{document}